\documentclass[10pt]{amsart}
\usepackage{amssymb}
\usepackage{amscd}
\usepackage[all]{xy}

\numberwithin{equation}{section}

\def\today{\number\day\space\ifcase\month\or   January\or February\or
   March\or April\or May\or June\or   July\or August\or September\or
   October\or November\or December\fi{}  \number\year}

\theoremstyle{definition}
\newtheorem{thm}{Theorem}[section]
\newtheorem{lem}[thm]{Lemma}
\newtheorem{prp}[thm]{Proposition}
\newtheorem{dfn}[thm]{Definition}
\newtheorem{cor}[thm]{Corollary}

\newtheorem{rmk}[thm]{Remark}
\newtheorem{ntn}[thm]{Notation}
\newtheorem{exa}[thm]{Example}
\newtheorem{pbm}[thm]{Problem}

\newtheorem{qst}[thm]{Question}

\newcommand{\beq}{\begin{equation}}
\newcommand{\eeq}{\end{equation}}
\newcommand{\beqr}{\begin{eqnarray*}}
\newcommand{\eeqr}{\end{eqnarray*}}
\newcommand{\bal}{\begin{align*}}
\newcommand{\eal}{\end{align*}}
\newcommand{\bei}{\begin{itemize}}
\newcommand{\eei}{\end{itemize}}
\newcommand{\limi}[1]{\lim_{{#1} \to \infty}}

\newcommand{\af}{\alpha}
\newcommand{\bt}{\beta}
\newcommand{\gm}{\gamma}
\newcommand{\dt}{\delta}
\newcommand{\ep}{\varepsilon}
\newcommand{\zt}{\zeta}
\newcommand{\et}{\eta}
\newcommand{\ch}{\chi}

\newcommand{\ld}{\lambda}
\newcommand{\sm}{\sigma}

\newcommand{\ph}{\varphi}
\newcommand{\ps}{\psi}
\newcommand{\rh}{\rho}
\newcommand{\om}{\omega}
\newcommand{\ta}{\tau}

\newcommand{\Dt}{\Delta}

\newcommand{\Q}{{\mathbb{Q}}}
\newcommand{\Z}{{\mathbb{Z}}}
\newcommand{\R}{{\mathbb{R}}}
\newcommand{\C}{{\mathbb{C}}}
\newcommand{\N}{{\mathbb{Z}}_{> 0}}
\newcommand{\Nz}{{\mathbb{Z}}_{\geq 0}}

\newcommand{\inv}{{\mathrm{inv}}}

\newcommand{\id}{{\mathrm{id}}}

\newcommand{\spn}{{\mathrm{span}}}
\newcommand{\card}{{\mathrm{card}}}
\newcommand{\Aut}{{\mathrm{Aut}}}

\newcommand{\tr}{{\operatorname{tr}}}

\newcommand{\dirlim}{\varinjlim}

\newcommand{\andeqn}{\,\,\,\,\,\, {\mbox{and}} \,\,\,\,\,\,}

\newcommand{\tfae}{the following are equivalent}
\newcommand{\ifo}{if and only if}

\newcommand{\ca}{C*-algebra}

\newcommand{\hm}{homomorphism}

\newcommand{\fd}{finite dimensional}

\newcommand{\ct}{continuous}
\newcommand{\cfn}{continuous function}

\newcommand{\rpn}{representation}

\newcommand{\sfm}{$\sm$-finite measure space}

\newcommand{\XBM}{(X, {\mathcal{B}}, \mu)}
\newcommand{\YCN}{(Y, {\mathcal{C}}, \nu)}
\newcommand{\cB}{{\mathcal{B}}}
\newcommand{\cC}{{\mathcal{C}}}
\newcommand{\LLp}{L (L^p (X, \mu))}

\newcommand{\OP}[2]{{\mathcal{O}}_{#1}^{#2}}
\newcommand{\MP}[2]{M_{#1}^{#2}}

\newcommand{\ov}{\overline}

\newcommand{\I}{\infty}
\newcommand{\E}{\varnothing}
\newcommand{\Lem}[1]{Lemma~\ref{#1}}
\newcommand{\Def}[1]{Definition~\ref{#1}}

\title[UHF and Cuntz algebras on $L^p$~spaces]{Simplicity
    of UHF and Cuntz algebras on $L^p$~spaces}

\author{N.~Christopher Phillips}

\date{14~September 2013}

\address{Department of Mathematics, University  of Oregon,
       Eugene OR 97403-1222, USA}
\email[]{ncp@darkwing.uoregon.edu}

\thanks{This material is based upon work supported by the
  US National Science Foundation under Grants
  DMS-0701076 and DMS-1101742.
  It was also partially supported by the Centre de Recerca
  Matem\`{a}tica (Barcelona) through a research visit conducted
  during 2011,
  and by the Research Institute for Mathematical Sciences
  of Kyoto University through a visiting professorship
  in 2011--2012.}

\begin{document}

\begin{abstract}
We prove that, for $1 \leq p < \infty$
and $d \in \{ 2, 3, 4, \ldots \},$
the $L^p$~analog ${\mathcal{O}}_d^p$
of the Cuntz algebra ${\mathcal{O}}_d$
is a purely infinite simple amenable Banach algebra.

The proof requires what we call the spatial $L^p$~UHF algebras,
which are analogs of UHF algebras acting on $L^p$~spaces.
As for the usual UHF C*-algebras,
they have associated supernatural numbers.
For fixed $p \in [1, \infty),$
we prove that any spatial $L^p$~UHF algebra is simple and amenable,
and that two such algebras are isomorphic
if and only if they have the same supernatural number
(equivalently, the same scaled ordered $K_0$-group).
For distinct $p_1, p_2 \in [1, \infty),$
we prove that no spatial $L^{p_1}$~UHF algebra
is isomorphic to any spatial $L^{p_2}$~UHF algebra.
\end{abstract}

\maketitle

\indent
Analogs~$\OP{d}{p}$ of the Cuntz algebras~${\mathcal{O}}_d,$
but acting on $L^p$~spaces,
were introduced in~\cite{Ph5}.
(See Definition~\ref{D:LPCuntzAlg} below.)
For $d \in \{ 2, 3, 4, \ldots \}$
and $p \in [1, \I) \setminus \{ 2 \},$
Theorem~8.7 of~\cite{Ph5} is a uniqueness theorem
for $\OP{d}{p},$
but simplicity of $\OP{d}{p}$ is not proved in~\cite{Ph5}.
The purpose of this paper is to prove that $\OP{d}{p}$
is purely infinite in a suitable sense, simple,
and
amenable as a Banach algebra.
As an intermediate step,
we define and study analogs of UHF algebras on $L^p$~spaces.
We prove simplicity and existence of a unique normalized trace
for a fairly large class of such algebras,
and amenability and a uniqueness theorem
for the standard examples (which we call spatial $L^p$~UHF algebras).
We also prove that two of our standard examples are isomorphic
\ifo{} they use the same value of~$p$
and they have the same scaled ordered $K_0$-group.
We mostly avoid the intricacies of K-theory
by using the supernatural number as our invariant.
In a separate paper~\cite{Ph-Lp2b},
we consider a somewhat more restrictive class than here.
Within this class,
for each $p \in [1, \I)$ and each possible $K_0$-group of a UHF algebra,
we exhibit uncountably many isomorphism classes of $L^p$~UHF algebras
with the given $K_0$-group.
We also show that, within this class,
amenability implies isomorphism with a spatial $L^p$~UHF algebra.

In the \ca{} case ($p = 2$),
simplicity for ${\mathcal{O}}_d$ is equivalent to uniqueness.
For $p \neq 2,$
we know of no method of getting either of simplicity or uniqueness
from the other.
To get simplicity from uniqueness,
one needs to know that a purported proper quotient
of~$\OP{d}{p}$
can again be represented on an $L^p$~space.
The closest we know to this is the result of~\cite{LM},
but for $p \neq 2$ this result does not suffice.
(For $p = 1,$ in fact it says nothing at all.)
To get uniqueness from simplicity,
we need to know that any continuous homomorphism
from a simple Banach algebra of operators
on an $L^p$~space to another one is isometric;
for a weaker form of uniqueness,
we at least need to know that such a homomorphism
has closed range.
Examples in~\cite{Ph-Lp2b} show that this is false
for general $L^p$~UHF algebras.

The proof of uniqueness in~\cite{Ph5}
is unrelated to the usual methods used in \ca{s}.
Indeed, it does not work for \ca{s}.
The proof of simplicity and pure infiniteness given here
follows the method of the original proof~\cite{Cu1} in the \ca{} case.
Much more care is needed,
and the proofs depend heavily on the theory of spatial \rpn{s}
developed in~\cite{Ph5},
but the basic outline is the same.

We do not consider $L^p$~analogs of general AF~algebras in this paper,
only $L^p$~analogs of UHF algebras,
since they are notationally simpler,
automatically simple,
avoid extra complications,
and suffice for the purposes here.
In particular,
the theory developed here and in~\cite{Ph-Lp2b}
seems inadequate to handle commutative $L^p$~analogs of AF algebras.
However,
$L^p$~analogs of AF algebras
seem to be interesting in their own right,
and will be further studied in~\cite{Vl}.

Section~\ref{Sec:Plm} contains preliminary material.
We recall some (not all) of the important facts from~\cite{Ph5}
which we use,
and we describe the $L^p$~analog of the minimal (spatial)
tensor product of \ca s.
In Section~\ref{Sec:CondExpt},
we give some material on conditional expectations on Banach algebras
and isometric actions of compact groups on Banach algebras.
This material is needed for technical purposes in the rest of the paper.
Section~\ref{Sec_N_UHF} introduces $L^p$~analogs of UHF algebras,
including the ``canonical'' class consisting of
the spatial $L^p$~UHF algebras.
It contains a uniqueness theorem for the spatial $L^p$~UHF algebras
and theorems on simplicity and uniqueness of the trace
for a larger class.
We also prove that spatial $L^p$~UHF algebras
are amenable in the Banach algebra sense.
In Section~\ref{Sec_Diffp}, we prove
that two spatial $L^p$~UHF algebras are isomorphic
\ifo{} they have the same supernatural number
and use the same choice of~$p.$
This material is not used later.
We leave several problem open.
In Section~\ref{Sec:Simplicity},
we prove simplicity, pure infiniteness,
and amenability of~$\OP{d}{p}.$

Scalars will always be~$\C$
(except in part of Section~\ref{Sec_Diffp},
where we need to quote results stated
for real scalars).
Presumably everything here works for real scalars,
except that the K-theory will be different.
(The proof, in~\cite{Ph5}, of the equivalence of some of the
conditions for a representation of a finite dimensional matrix algebra
to be spatial,
depends on complex scalars,
although the result might still be valid in the real case.
The affected implications are probably not needed here.)

We will use the following notation throughout.

\begin{ntn}\label{N_2Y15_TrInv}
For $d \in \N,$
we let $\tr \colon M_d \to \C$ be the normalized trace,
that is,
\[
\tr \big( (b_{j, k})_{j, k = 1}^d \big)
 = \frac{1}{d} \sum_{j = 1}^d b_{j, j}.
\]
If we want to make the matrix size explicit,
we write $\tr_d \colon M_d \to \C.$
Also, for any unital algebra~$A,$
we let $\inv (A)$ denote the group of invertible elements of~$A.$
If $E$ and $F$ are Banach spaces,
we write $L (E, F)$ for the space of bounded linear operators
from $E$ to~$F,$
and $L (E)$ for $L (E, E).$
For $a \in L (E, F),$
we denote its range by ${\mathrm{ran}} (a).$
\end{ntn}

A unital representation of a unital complex algebra on a Banach space
has the obvious meaning.
(See Definition~2.8 of~\cite{Ph5}.)
We take all algebras and representations to be unital.

As in~\cite{Ph5},
isomorphisms of Banach spaces and Banach algebras
are not assumed to be isometric.

We are grateful to
Narutaka Ozawa for suggesting that we consider amenability
of our algebras and supplying the reference~\cite{Rs}.
We are especially grateful to Bill Johnson
for extensive discussions about Banach spaces,
to Guillermo Corti\~{n}as
and Mar\'{\i}a Eugenia Rodr\'{\i}guez,
who carefully read an early draft
and whose comments led to numerous
corrections and improvements,
and to Maria Grazia Viola, whose reading of a later
draft led to further corrections and improvements.
Some of this work was carried out during a visit to
the Instituto Superior T\'{e}cnico,
Universidade T\'{e}cnica de Lisboa,
and during an extended research visit to
the Centre de Recerca Matem\`{a}tica (Barcelona).
I also thank the Research Institute for Mathematical Sciences
of Kyoto University for its support through a visiting professorship.
I am grateful to all these institutions for their hospitality.

\section{Preliminaries}\label{Sec:Plm}

\indent
In this section, we recall the definitions
of the Banach algebras~$\MP{d}{p}$
(algebraically just the algebra $M_d$ of $d \times d$ complex matrices)
and~$\OP{d}{p}$ (a suitable completion of the Leavitt algebra~$L_d$).
We also give definitions of spatial \rpn s of these algebras.
We then discuss the $L^p$~spatial tensor product of
closed subalgebras $A \subset \LLp$ and $B \subset L (L^p (Y, \nu)).$

We start with~$\MP{d}{p}.$

\begin{ntn}\label{N:FDP}
For any set~$S,$ we give $l^p (S)$ the usual
meaning (using counting measure on~$S$),
and we take (as usual) $l^p = l^p (\N).$
For $d \in \N$ and $p \in [1, \I],$
we let $l^p_d = l^p \big( \{1, 2, \ldots, d \} \big).$
We further let $\MP{d}{p} = L \big( l_d^p \big)$
with the usual operator norm,
and we algebraically identify $\MP{d}{p}$ with the algebra $M_d$ of
$d \times d$ complex matrices in the standard way.
\end{ntn}

We warn of a notational conflict.
Many articles on Banach spaces use $L_p (X, \mu)$
rather than $L^p (X, \mu),$
and use $l_p^d$ for what we call~$l_d^p.$
Our convention is chosen
to avoid conflict with the standard notation
for the Leavitt algebra $L_d$ of Definition~\ref{D:Leavitt} below.

\begin{dfn}\label{D_2Y14_SpMd}
Let $d \in \{ 2, 3, 4, \ldots \},$
let $p \in [1, \I) \setminus \{ 2 \},$
let $\XBM$ be a \sfm,
and let $\rh \colon M_d \to \LLp$ be a \rpn.
We say that $\rh$ is {\emph{spatial}} if $\rh$
is contractive for the norm of Notation~\ref{N:FDP},
that is, $\| \rh (a) \| \leq \| a \|$ for all $a \in M_d.$
\end{dfn}

This is not the original definition of a spatial \rpn.
Instead, it is the equivalent condition
in Theorem 7.2(4) of~\cite{Ph5}.
We refer to Section~7 of~\cite{Ph5}
for motivation for this definition,
and note that Theorem 7.2 of~\cite{Ph5} gives a number
of equivalent conditions for a \rpn{} of~$M_d$
to be spatial.
One of them is important enough to be stated here for reference.

\begin{thm}\label{T_2Y17Iso}
Let the notation be as in Definition~\ref{D_2Y14_SpMd}.
Then $\rh$ is spatial \ifo{} $\rh$ is isometric.
\end{thm}

\begin{proof}This is the equivalence of conditions
(3) and~(4) of Theorem 7.2 of~\cite{Ph5}.
\end{proof}

For $p = 2,$
contractive representations need not be spatial
in the sense of Definition~7.1 of~\cite{Ph5}.
We use the following substitute.
It was pointed to us by Narutaka Ozawa,
and the reference was supplied by David Blecher.

\begin{lem}[Proposition A.5.8 of~\cite{BM}]\label{L_3717_CSt}
Let $A$ and $B$ be \ca{s},
and let $\ph \colon A \to B$ be an algebra \hm.
Then $\ph$ is contractive \ifo{} $\ph$ is a *-\hm.
\end{lem}

We now turn to~$\OP{d}{p}.$
It is the operator norm closure of what we call in~\cite{Ph5}
a spatial representation on an $L^p$~space
of the Leavitt algebra~$L_d.$

\begin{dfn}\label{D:Leavitt}
Let $d \in \{ 2, 3, 4, \ldots \}.$
We define the {\emph{Leavitt algebra}} $L_d$
to be the universal unital complex associative algebra on
generators $s_1, s_2, \ldots, s_d, t_1, t_2, \ldots, t_d$
satisfying the relations:
\begin{enumerate}
\item\label{D:Leavitt_1}
$t_j s_j = 1$ for $j \in \{ 1, 2, \ldots, d \}.$
\item\label{D:Leavitt_2}
$t_j s_k = 0$ for $j, k \in \{ 1, 2, \ldots, d \}$ with $j \neq k.$
\item\label{D:Leavitt_3}
$\sum_{j = 1}^d s_j t_j = 1.$
\end{enumerate}
\end{dfn}

\begin{dfn}\label{D_2Y14_SpRep}
Let $d \in \{ 2, 3, 4, \ldots \},$
let $p \in [1, \I) \setminus \{ 2 \},$
let $\XBM$ be a \sfm,
and let $\rh \colon L_d \to \LLp$ be a \rpn.
We say that $\rh$ is {\emph{spatial}} if the following
two conditions are satisfied:
\begin{enumerate}
\item\label{D_2Y14_SpRep_MnSp}
Identifying
\[
B = \spn \big( \{ s_j t_k \colon j, k = 1, 2, \ldots, d \} \big)
\]
with $M_d$ in the usual way
(see the case $m = 1$ of Lemma~\ref{L:mSum} for details),
and then giving $B$ the norm from Notation~\ref{N:FDP},
the restriction $\rh |_B \colon B \to \LLp$ is contractive.
\item\label{D_2Y14_SpRep_Iso}
For $j = 1, 2, \ldots, d,$
we have $\| \rh (s_j) \| \leq 1$ and $\| \rh (t_j) \| \leq 1.$
\end{enumerate}
\end{dfn}

Again, this is not the original definition of a spatial \rpn.
Instead, it is the combination of Theorem 7.7(5)
and Theorem 7.2(4) of~\cite{Ph5}.
Theorems 7.2 and~7.7 of~\cite{Ph5} together
give many other equivalent conditions for a \rpn{} to be spatial.

Spatial \rpn s always exist (Lemma~7.5 of~\cite{Ph5}),
and the norm on $L_d$ defined by $x \mapsto \| \rh (x) \|$
% for $x \in L_d,$
is independent of the choice of the spatial \rpn~$\rh$
(Theorem~8.7 of~\cite{Ph5}).
Therefore the following definition makes sense.

\begin{dfn}[Definition~8.8 of~\cite{Ph5}]\label{D:LPCuntzAlg}
Let $d \in \{ 2, 3, 4, \ldots \}$
and let $p \in [1, \I) \setminus \{ 2 \}.$
We define $\OP{d}{p}$
to be the completion of $L_d$ in the norm
$a \mapsto \| \rh (a) \|$ for any spatial \rpn{} $\rh$
of $L_d$ on a space of the form $L^p (X, \mu)$
for a \sfm{} $\XBM.$
We write $s_j$ and $t_j$ for the elements in $\OP{d}{p}$
obtained as the images of the elements with the same names
in~$L_d,$
as in Definition~\ref{D:Leavitt}.

We take $\OP{d}{2}$ to be the usual Cuntz algebra~${\mathcal{O}}_d,$
with standard generating isometries $s_1, s_2, \ldots, s_d,$
and with $t_j = s_j^*$ for $j = 1, 2, \ldots, d.$
\end{dfn}

For reference, we restate here Theorem~2.16 of~\cite{Ph5}.
(This theorem is really a compilation of results from other sources.)

\begin{thm}\label{T-LpTP}
Let $\XBM$ and $\YCN$ be \sfm s.
Let $p \in [1, \I).$
Write $L^p (X, \mu) \otimes_p L^p (Y, \nu)$
for the Banach space completed tensor product
$L^p (X, \mu) {\widetilde{\otimes}}_{\Dt_p} L^p (Y, \nu)$
defined in~7.1 of~\cite{DF}.
Then there is a unique isometric isomorphism
\[
L^p (X, \mu) \otimes_p L^p (Y, \nu)
  \cong L^p (X \times Y, \, \mu \times \nu)
\]
which identifies,
for every $\xi \in L^p (X, \mu)$ and $\et \in L^p (Y, \nu),$
the element $\xi \otimes \et$ with the function
$(x, y) \mapsto \xi (x) \et (y)$ on $X \times Y.$
Moreover:
\begin{enumerate}
\item\label{T-LpTP-0}
Under the identification above,
the linear span of all $\xi \otimes \et,$
for $\xi \in L^p (X, \mu)$ and $\et \in L^p (Y, \nu),$
is dense in $ L^p (X \times Y, \, \mu \times \nu).$
\item\label{T-LpTP-1}
$\| \xi \otimes \et \|_p = \| \xi \|_p \| \et \|_p$
for all $\xi \in L^p (X, \mu)$ and $\et \in L^p (Y, \nu).$
\item\label{T-LpTP-2}
The tensor product $\otimes_p$ is commutative and associative.
\item\label{T-LpTP-3a}
Let
\[
(X_1, \cB_1, \mu_1), \,\,\,\,\,\,
(X_2, \cB_2, \mu_2), \,\,\,\,\,\,
(Y_1, \cC_1, \nu_1),
\andeqn
(Y_2, \cC_2, \nu_2)
\]
be \sfm s.
Let
\[
a \in L \big( L^p (X_1, \mu_1), \, L^p (X_2, \mu_2) \big)
\andeqn
b \in L \big( L^p (Y_1, \nu_1), \, L^p (Y_2, \nu_2) \big).
\]
Then there exists a unique
\[
c \in L \big( L^p (X_1 \times Y_1, \, \mu_1 \times \nu_1),
          \, L^p (X_2 \times Y_2, \, \mu_2 \times \nu_2) \big)
\]
such that, making the identification above,
$c (\xi \otimes \et) = a \xi \otimes b \et$
for all $\xi \in L^p (X_1, \mu_1)$ and $\et \in L^p (Y_1, \nu_1).$
We call this operator $a \otimes b.$
\item\label{T-LpTP-3b}
The operator $a \otimes b$ of~(\ref{T-LpTP-3a})
satisfies $\| a \otimes b \| = \| a \| \cdot \| b \|.$
\item\label{T-LpTP-4}
The tensor product of operators defined in~(\ref{T-LpTP-3a})
is associative, bilinear, and satisfies (when the domains
are appropriate)
$(a_1 \otimes b_1) (a_2 \otimes b_2) = a_1 a_2 \otimes b_1 b_2.$
\end{enumerate}
\end{thm}

\begin{dfn}\label{D:LpTensorPrd}
Let $\XBM$ and $\YCN$ be \sfm s,
and let $p \in [1, \I).$
Let $A \subset \LLp$ and $B \subset L (L^p (Y, \nu))$
be closed subalgebras.
We define
$A \otimes_p B
 \subset L \big( L^p (X \times Y, \, \mu \times \nu) \big)$
to be the closed linear span of all $a \otimes b$
with $a \in \LLp$ and $b \in L (L^p (Y, \nu)).$
We call $A \otimes_p B$ the {\emph{$L^p$ operator tensor product}}
of $A$ and~$B.$
We call an element $a \otimes b$
an {\emph{elementary tensor}} of operators.
\end{dfn}

We warn that this use of $\otimes_p$ is quite different from
the tensor product in Theorem~\ref{T-LpTP} which gives
$L^p (X, \mu) \otimes_p L^p (Y, \nu)
   \cong L^p (X \times Y, \, \mu \times \nu).$
No confusion should arise.

When $p = 2$ and $A$ and $B$ are closed under $a \mapsto a^*$
(that is, they are \ca s),
then $A \otimes_p B$ is the minimal C*~tensor product of $A$ and~$B.$
See Section~6.3 of~\cite{Mr},
where it is called the spatial tensor product.
It does not depend on how $A$ and~$B$ are represented.
For general~$p,$
and even for nonselfadjoint algebras when $p = 2,$
our tensor product does depend on how the algebras are represented.
Examples will be given elsewhere.
We thus do not try to define a tensor norm and a completed tensor product
independently of the representations of $A$ and~$B,$
except in very special cases.

We now give some properties of $A \otimes_p B.$
We will not use injectivity in~(\ref{P:TProdOfAlgs-3}),
but it is the sort of thing that one ought to know.
We also mostly do not explicitly refer to the others.

\begin{prp}\label{P:TProdOfAlgs}
Let $p \in [1, \I).$
The $L^p$ operator tensor product of algebras
(\Def{D:LpTensorPrd})
has the following properties:
\begin{enumerate}
\item\label{P:TProdOfAlgs-0}
$A \otimes_p B$ is a Banach algebra.
\item\label{P:TProdOfAlgs-1}
Associativity:
If $\XBM,$ $\YCN,$ and $(Z, {\mathcal{D}}, \ld)$
are \sfm s,
and $A \subset \LLp,$ $B \subset L (L^p (Y, \nu)),$
and $C \subset L (L^p (Z, \ld))$ are closed subalgebras,
then $(A \otimes_p B) \otimes_p C = A \otimes_p (B \otimes_p C)$
as subalgebras of
$L \big( L^p (X \times Y \times Z, \, \mu \times \nu \times \ld) \big).$
\item\label{P:TProdOfAlgs-2}
If $A$ and $B$ are unital,
then $A \otimes_p B$ is unital, with identity $1 \otimes 1,$
and $a \mapsto a \otimes 1$ and $b \mapsto 1 \otimes b$
are unital isometric \hm s from $A$ and $B$ to $A \otimes_p B.$
\item\label{P:TProdOfAlgs-3}
Let $A \otimes_{\mathrm{alg}} B$ be the usual algebraic tensor product.
Then there is a canonical algebra \hm{}
$\ph \colon A \otimes_{\mathrm{alg}} B \to A \otimes_p B$
sending $a \otimes b \in A \otimes_{\mathrm{alg}} B$
to $a \otimes b \in A \otimes_p B$ as defined above,
and this \hm{} is injective.
% \item\label{P:TProdOfAlgs-4}
% \item\label{P:TProdOfAlgs-5}
\end{enumerate}
\end{prp}

\begin{proof}
% [Proof of Proposition~\ref{P:TProdOfAlgs}]
Part~(\ref{P:TProdOfAlgs-0}) follows from the fact that the
linear span of the elementary tensors of operators is
closed under multiplication.
(This uses Theorem~\ref{T-LpTP}(\ref{T-LpTP-4}).)
Part~(\ref{P:TProdOfAlgs-1}) is easily proved by repeated
applications of Theorem~\ref{T-LpTP}(\ref{T-LpTP-0}),
together with the fact that
$(\xi \otimes \et) \otimes \zt = \xi \otimes (\et \otimes \zt).$
Part~(\ref{P:TProdOfAlgs-2}) follows from parts (\ref{T-LpTP-3b}) and~(\ref{T-LpTP-4})
of Theorem~\ref{T-LpTP}.

We prove~(\ref{P:TProdOfAlgs-3}).
Existence of $\ph$ is clear,
as is the fact that $\ph$ is an algebra \hm.

It remains to prove injectivity.

Let $c \in A \otimes_{\mathrm{alg}} B$ satisfy $\ph (c) = 0.$
We prove that $c = 0.$
Choose
\[
m \in \N,
\,\,\,\,\,\,
r_1, r_2, \ldots, r_m \in A,
\andeqn
s_1, s_2, \ldots, s_m \in B,
\]
such that $c = \sum_{j = 1}^m r_j \otimes s_j.$
Let $a_1, a_2, \ldots, a_n \in A$
form a basis for $\spn (r_1, r_2, \ldots, r_m).$
For $j = 1, 2, \ldots, m$ and $k = 1, 2, \ldots, n$
choose $\ld_{j, k} \in \C$
such that $r_j = \sum_{k = 1}^n \ld_{j, k} a_k.$
Then set $b_k = \sum_{j = 1}^m \ld_{j, k} s_j.$
This gives $c = \sum_{k = 1}^n a_k \otimes b_k.$

% If $b_k = 0$ for $k = 1, 2, \ldots, n,$
% then $c = 0,$ and we are done.

Let $q \in (1, \I]$ satisfy $\frac{1}{p} + \frac{1}{q} = 1.$
Let $E$ be the vector space of all
bounded bilinear maps $L^p (X, \mu) \times L^q (X, \mu) \to \C,$
and let $F$ be the vector space of all
bounded bilinear maps $L^p (Y, \nu) \times L^q (Y, \nu) \to \C.$
Define $R \colon \LLp \to E$ and $S \colon L (L^p (Y, \nu)) \to F$
by
\[
R (a) (\xi_1, \xi_2)
 = \int_X \xi_2 \cdot a (\xi_1) \, d \mu
\andeqn
S (b) (\et_1, \et_2)
 = \int_Y \et_2 \cdot b (\et_1) \, d \nu
\]
for $a \in \LLp,$
$\xi_1 \in L^p (X, \mu),$
$\xi_2 \in L^q (X, \mu),$
$b \in L (L^p (Y, \nu)),$
$\et_1 \in L^p (Y, \nu),$
and $\et_2 \in L^q (Y, \nu).$
Since $L^q (X, \mu)$ is the dual of $L^p (X, \mu),$
the Hahn-Banach Theorem implies that $R$ is injective.
Similarly $S$ is injective.
Further,
for $\et_1 \in L^p (Y, \nu)$ and $\et_2 \in L^q (Y, \nu),$
define
$T_{\et_1, \et_2} \colon L (L^p (X \times Y, \, \mu \times \nu)) \to E$
by
\[
T_{\et_1, \et_2} (a) (\xi_1, \xi_2)
 = \int_{X \times Y} \xi_2 (x) \et_2 (y) a (\xi_1 \otimes \et_1) (x, y)
   \, d \mu (x) d \nu (y)
\]
for $a \in L (L^p (X \times Y, \, \mu \times \nu)),$
$\xi_1 \in L^p (X, \mu),$
and $\xi_2 \in L^q (X, \mu).$
Using Fubini's Theorem at the second step,
for $\xi_1 \in L^p (X, \mu)$
and $\xi_2 \in L^q (X, \mu),$
we have
\begin{align*}
0 & = T_{\et_1, \et_2} ( \ph (c)) (\xi_1, \xi_2)
    = \sum_{k = 1}^n
       \left( \int_Y \et_2 \cdot b_k (\et_1) \, d \nu \right)
       \left( \int_X \xi_2 \cdot a_k (\xi_1) \, d \mu \right)
     \\
  & = R \left( \sum_{k = 1}^n S (b_k) (\et_1, \et_2) a_k \right)
         (\xi_1, \xi_2).
\end{align*}
Since $R$ is injective and $\xi_1$ and $\xi_2$ are arbitrary,
we get $\sum_{k = 1}^n S (b_k) (\et_1, \et_2) a_k = 0.$
Since $a_1, a_2, \ldots, a_n$ are linearly independent,
it follows that $S (b_k) (\et_1, \et_2) = 0$
for $k = 1, 2, \ldots, n.$
Since $S$ is injective and $\et_1$ and $\et_2$ are arbitrary,
we get $b_k = 0$
for $k = 1, 2, \ldots, n.$
Therefore $c = 0.$
This completes the proof.
\end{proof}

The $p$-tensor product of closed
subalgebras of the bounded operators on $L^p$-spaces is not functorial
for contractive \hm{s}.
Examples will be given elsewhere.
(It probably is true that,
if $A_j \subset L (L^p (X_j, \mu_j))$
and $B_j \subset L (L^p (Y_j, \nu_j)),$
for $j = 1, 2,$
are closed subalgebras,
and if $\ph_j \colon A_j \to B_j,$
for $j = 1, 2,$
is a \hm{} which is completely bounded in a suitable sense,
then there is a completely bounded \hm{}
$\ph \colon A_1 \otimes_p A_2 \to B_1 \otimes_p B_2$
such that $\ph (a_1 \otimes a_2) = \ph_1 (a_1) \otimes \ph_2 (a_2)$
for all $a_1 \in A_1$ and $a_2 \in A_2.$)
We can, however, say something about the tensor product
of spatial representations,
and get some other information.
We give explicit details so that it is clear that
we are not making unjustified assumptions.

\begin{lem}\label{L_2Z29_PermFact}
Let $p \in [1, \I),$
let $n \in \N,$
and for $k = 1, 2, \ldots, n$
let $(X_k, {\mathcal{B}}_k, \mu_k)$
be a \sfm{}
and let $A_k \subset L (L^p (X_k, \mu_k))$
be a closed subalgebra.
Let $\sm$ be a permutation of $\{ 1, 2, \ldots, n \}.$
Then the permutation of tensor factors
\[
a_1 \otimes a_2 \otimes \cdots \otimes a_n \mapsto
  a_{\sm (1)} \otimes a_{\sm (2)} \otimes \cdots \otimes a_{\sm (n)}
\]
extends to an isometric isomorphism
\[
A_1 \otimes A_2 \otimes \cdots \otimes A_n \cong
  A_{\sm (1)} \otimes A_{\sm (2)} \otimes \cdots \otimes A_{\sm (n)}.
\]
\end{lem}

\begin{proof}
Set $Y = \prod_{k = 1}^n X_k$ and $Z = \prod_{k = 1}^n X_{\sm (k)},$
equipped with the obvious product measures $\nu$ and~$\ld.$
Let
$h \colon Y \to Z$
be
\[
h (x_1, x_2, \ldots, x_n)
  = (x_{\sm (1)}, x_{\sm (2)}, \ldots, x_{\sm (n)} )
\]
for $(x_1, x_2, \ldots, x_n) \in Y.$
Then $h$ is a measure preserving bijection,
so composition with $h$ defines an isometric isomorphism
$u \in L \big( L^p (Z, \ld), \, L^p (Y, \nu) \big).$
The required isomorphism is $a \mapsto u a u^{-1}.$
\end{proof}

\begin{lem}\label{L_2Y21MpTRep}
Let $p \in [1, \I) \setminus \{ 2 \},$
let $m, n \in \N,$
let $\XBM$ and $\YCN$ be \sfm{s},
and let $\rh \colon M_m \to \LLp$
and $\sm \colon M_n \to L (L^p (Y, \nu))$
be spatial \rpn{s}.
Identify $\C^m \otimes \C^n$ with $\C^{m n}$
via an isomorphism which sends tensor products of standard basis vectors
to standard basis vectors,
and use this isomorphism to identify
$M_m \otimes M_n = L (\C^m \otimes \C^n)$
with~$M_{m n} = L (\C^{m n}).$
Define a \rpn{}
\[
\rh \otimes \sm \colon M_{m n}
  \to L \big( L^p (X \times Y, \, \mu \times \nu) \big)
\]
by $(\rh \otimes \sm) (a \otimes b) = \rh (a) \otimes \sm (b)$
for $a \in M_m$ and $b \in M_n.$
Then $\rh \otimes \sm$ is spatial.
\end{lem}

\begin{proof}
This is immediate from the original definition
of a spatial representation
(Definition~7.1 of~\cite{Ph5})
and Lemma 6.20 of~\cite{Ph5}.
\end{proof}

\begin{cor}\label{C_2Y21MpTens}
Let $p \in [1, \I)$
and let $m, n \in \N.$
Then the identification of $M_m \otimes M_n$ with $M_{m n}$
used in Lemma~\ref{L_2Y21MpTRep}
induces an isometric isomorphism
$\MP{m}{p} \otimes_p \MP{n}{p} \cong \MP{m n}{p}.$
Moreover,
if $\XBM$ and $\YCN$ are \sfm{s},
and $\rh \colon \MP{m}{p} \to \LLp$
and $\sm \colon \MP{n}{p} \to L (L^p (Y, \nu))$
are isometric unital \rpn{s},
then the \rpn{} $\rh \otimes \sm$
of Lemma~\ref{L_2Y21MpTRep} is isometric.
\end{cor}

\begin{proof}
For $p \neq 2,$
this is immediate from Lemma~\ref{L_2Y21MpTRep}
and the fact (Theorem~\ref{T_2Y17Iso})
that a \rpn{} of $M_d$ is spatial \ifo{}
it is isometric on $\MP{d}{p}.$
For $p = 2,$
use Lemma~\ref{L_3717_CSt}.
\end{proof}

\begin{lem}\label{L_2Z29_Lower}
Let $p \in [1, \I),$
let $m, n \in \N,$
let $\XBM$ and $\YCN$ be \sfm{s},
and let $\rh \colon M_m \to \LLp$
and $\sm \colon M_n \to L (L^p (Y, \nu))$
be \rpn{s} (not necessarily spatial).
Identify $\MP{m}{p} \otimes_p \MP{n}{p}$ with $\MP{m n}{p}$
as in Corollary~\ref{C_2Y21MpTens},
and define a \rpn{}
\[
\rh \otimes_p \sm \colon \MP{m n}{p}
  \to L \big( L^p (X \times Y, \, \mu \times \nu) \big)
\]
by $(\rh \otimes_p \sm) (a \otimes b) = \rh (a) \otimes \sm (b)$
for $a \in M_m$ and $b \in M_n.$
Then $\| \rh \otimes_p \sm \| \geq \| \rh \| \cdot \| \sm \|.$
\end{lem}

The \hm{} $\rh \otimes_p \sm$ is defined on all of
$\MP{m}{p} \otimes_p \MP{n}{p},$
and bounded,
because $\MP{m}{p} \otimes_p \MP{n}{p}$ is \fd.
Allowing for the possibility that the tensor
product of the \hm{s} is unbounded
and defined only on the algebraic tensor product,
an analogous statement is true for more general domains.

We suspect that the inequality can't be replaced
by equality, but we don't have an example.

\begin{proof}[Proof of Lemma~\ref{L_2Z29_Lower}]
Let $\ep > 0.$
Choose $a \in \MP{m}{p}$ and $b \in \MP{n}{p}$
such that
\[
\| a \| = \| b \| = 1,
\,\,\,\,\,\,
\| \rh (a) \| > \| \rh \| - \ep,
\andeqn
\| \sm (b) \| > \| \sm \| - \ep.
\]
Then $\| a \otimes b \| = 1$ by Theorem~\ref{T-LpTP}(\ref{T-LpTP-3b}),
and
\[
\| (\rh \otimes_p \sm) (a \otimes b) \|
 > (\| \rh \| - \ep) (\| \sm \| - \ep),
\]
again by Theorem~\ref{T-LpTP}(\ref{T-LpTP-3b}).
Since $\ep > 0$ is arbitrary,
the result follows.
\end{proof}

\section{Conditional expectations and group actions}\label{Sec:CondExpt}

\indent
Our simplicity proofs depend heavily on the Banach algebra version
of a conditional expectation.
In the proof of simplicity of~$\OP{d}{p},$
the gauge action will also play a key role.
In this section,
we develop the properties of conditional expectations
and group actions which we use.

We only consider conditional expectations of norm one.
Doubtless the others are interesting and important for some purposes,
but the applications in this paper require the norm one condition.

\begin{dfn}\label{D:CondExpt}
Let $B$ be a unital Banach algebra,
and let $A \subset B$ be a closed subalgebra such that $1 \in A.$
A {\emph{Banach conditional expectation}} from $B$ to $A$
is a map $E \colon B \to A$ satisfying the following properties:
\begin{enumerate}
\item\label{D:CondExpt-1}
$E$ is linear.
\item\label{D:CondExpt-2}
$\| E \| = 1.$
\item\label{D:CondExpt-3}
$E (a) = a$ for all $a \in A.$
\item\label{D:CondExpt-4}
For all $a, c \in A$ and $b \in B,$
we have $E (a b c) = a E (b) c.$
% \item\label{D:CondExpt-5}
% \item\label{D:CondExpt-6}
\end{enumerate}
\end{dfn}

Examples are most easily obtained by averaging over group actions.

\begin{dfn}\label{D:Aut}
Let $B$ be a Banach algebra.
An {\emph{automorphism}} of~$B$
is a \ct{} linear bijection $\ph \colon B \to B$
such that $\ph (a b) = \ph (a) \ph (b)$ for all $a, b \in B.$
(Continuity of $\ph^{-1}$ is automatic,
by the Open Mapping Theorem.)
We call $\ph$ an {\emph{isometric automorphism}}
if in addition $\| \ph (a) \| = \| a \|$ for all $a \in B.$
We let $\Aut (B)$ denote the set of automorphisms of~$B.$
\end{dfn}

\begin{dfn}\label{D:Action}
Let $B$ be a Banach algebra,
and let $G$ be a topological group.
An {\emph{action of $G$ on~$B$}}
is a \hm{} $g \mapsto \af_g$ from $G$ to $\Aut (B)$
such that for every $a \in B,$
the map $g \mapsto \af_g (a)$ is \ct{} from $G$ to~$B.$
We say that the action is {\emph{isometric}} if each $\af_g$ is.
\end{dfn}

\begin{prp}\label{P:Avg}
Let $B$ be a unital Banach algebra,
let $G$ be a compact group with normalized Haar measure~$\nu,$
and let $\af \colon G \to \Aut (B)$ be an isometric action.
Let $B^G$ be the fixed point algebra,
\[
B^G = \big\{ a \in B \colon
          {\mbox{$\af_g (a) = a$ for all $g \in G$}} \big\}.
\]
Take Banach space valued integration as in Section~B.1
in Appendix~B of~\cite{Wl}.
Then the formula
\[
E (a) = \int_G \af_g (a) \, d \nu (g)
\]
defines a Banach conditional expectation from $B$ to~$B^G.$
\end{prp}

Since we are integrating continuous functions
with respect to Haar measure,
the presentation of vector valued integration
in Section~1.5 of~\cite{Wl} suffices.

\begin{proof}[Proof of Proposition~\ref{P:Avg}]
All the properties required of a Banach conditional expectation in
\Def{D:CondExpt}
are easily checked.
\end{proof}

\begin{rmk}\label{R_2Y16_AvgBddAct}
In Proposition~\ref{P:Avg},
suppose we drop the assumption that the action is isometric.
Since $G$ is compact,
we have $\sup_{g \in G} \| \af_g (a) \| < \I$ for all $a \in A.$
The Uniform Boundedness Principle therefore implies that
$\sup_{g \in G} \| \af_g \| < \I.$
The same proof then shows that $E$ satisfies the conditions
of \Def{D:CondExpt},
except that condition~(\ref{D:CondExpt-2})
must be replaced by
$\| E \| \leq \sup_{g \in G} \| \af_g \| < \I.$
\end{rmk}

\begin{exa}\label{E_2Y16_HmIsE}
Let $A$ be a commutative unital Banach algebra,
let $X$ be its maximal ideal space,
and let $\mu$ be a probability measure on~$X.$
Define $E \colon A \to \C \cdot 1 \subset A$
by $E (a) = \left( \int_X \om (a) \, d \mu (\om) \right) \cdot 1$
for $a \in A.$
Then $E$ is a Banach conditional expectation.
For the norm estimate,
one uses $\| \om \| \leq 1$ for all $\om \in X.$

In particular, if $\om \colon A \to \C$ is a multiplicative
linear functional,
then $a \mapsto \om (a) \cdot 1$
is a Banach conditional expectation.
\end{exa}

Our next example is the conditional expectation
from matrices over $A$ to $A$ itself coming from
the usual trace on a matrix algebra.
We need some preparatory work,
enabling us to show that the condition $\| E \| \leq 1$
is satisfied for many choices of the norm on matrices over~$A,$
and giving an important additional property.

% We now set up machinery which identifies some norms on $M_d$
% for which we get a Banach conditional expectation.
Unless otherwise specified,
the norm on $M_d$ itself will be as given
for $\MP{d}{p}$ in Notation~\ref{N:FDP};
other norms will be given in terms of the norm on
the image of $M_d$ under a \rpn.

Definition~\ref{D_2Y15_EnId}(\ref{D_2Y15_EnId_Compat}) below
is a normalization condition.
It is included for use in~\cite{Ph-Lp2b},
and will not be needed in this paper.

\begin{dfn}\label{D_2Y15_EnId}
Let $\XBM$ be a \sfm,
let $p \in [1, \I),$
and let $d \in \N.$
Let $\rh \colon M_d \to \LLp$
be a unital representation.
\begin{enumerate}
\item\label{D_2Y15_EnId_EI}
We say that
{\emph{$\rh$~has enough isometries}}
if there is a finite subgroup $G \subset \inv (M_d)$
whose natural action on $\C^d$ is irreducible
and such that $\| \rh (g) \| = 1$ for all $g \in G.$
\item\label{D_2Y15_EnId_Loc}
We say that
{\emph{$\rh$~locally has enough isometries}}
if there is an at most countable partition
$X = \coprod_{i \in I} X_{i}$
into measurable subsets $X_{i}$ such that $\mu (X_{i}) > 0$
for $i \in I,$
such that the subspace $L^p (X_{i}, \mu) \subset L^p (X, \mu)$
is $\rh$-invariant for every $i \in I,$
and such that for every $i \in I$
the \rpn{} $x \mapsto \rh (x) |_{L^p (X_{i}, \mu)}$
has enough isometries.
\item\label{D_2Y15_EnId_Compat}
If $\rh$~locally has enough isometries,
we say that {\emph{$\rh$~dominates the spatial representation}}
if
we can choose the partition
$X = \coprod_{i \in I} X_{i}$ in~(\ref{D_2Y15_EnId_Loc})
so that there is $i_0 \in I$
for which
the \rpn{} $x \mapsto \rh (x) |_{L^p (X_{i_0}, \mu)}$
is spatial.
\end{enumerate}
\end{dfn}

For $p = 2,$
in Definition \ref{D_2Y15_EnId}(\ref{D_2Y15_EnId_Compat})
we use spatial representations as in Definition~7.1 of~\cite{Ph5}.
Up to unitary equivalence,
however,
we may as well just ask for *-representations.

\begin{lem}\label{L_2Y15_IrrEq}
Let $d \in \N,$
and let $G \subset \inv (M_d)$ be a finite subgroup.
Then \tfae:
\begin{enumerate}
\item\label{L_2Y15_IrrEq_Irr}
The natural action of $G$ on $\C^d$ is irreducible.
\item\label{L_2Y15_IrrEq_Cnt}
If $a \in M_d$ commutes with every element of~$G,$
then $a \in \C \cdot 1.$
\item\label{L_2Y15_IrrEq_Scl}
For all $a \in M_d,$
the element $\sum_{g \in G} g a g^{-1}$ is a scalar multiple of~$1.$
\item\label{L_2Y15_IrrEq_Tr}
For all $a \in M_d,$
we have
\[
\frac{1}{\card (G)} \sum_{g \in G} g a g^{-1} = \tr (a) \cdot 1.
\]
\end{enumerate}
\end{lem}

\begin{proof}
The equivalence of (\ref{L_2Y15_IrrEq_Irr})
and~(\ref{L_2Y15_IrrEq_Cnt})
is well known.

Assume~(\ref{L_2Y15_IrrEq_Cnt}),
and let $a \in M_d.$
Set $b = \sum_{g \in G} g a g^{-1}.$
Then $b$ commutes with every element of~$G,$
so is a scalar.
This is~(\ref{L_2Y15_IrrEq_Scl}).

Now assume~(\ref{L_2Y15_IrrEq_Scl}),
and let $a \in M_d.$
Set
\[
b = \frac{1}{\card (G)} \sum_{g \in G} g a g^{-1}.
\]
We have $\tr (g a g^{-1}) = \tr (a)$ for all $g \in G.$
By hypothesis, there is $\ld \in \C$ such that
$b = \ld \cdot 1.$
Now
\[
\ld = \tr (b)
    = \frac{1}{\card (G)} \sum_{g \in G} \tr (g a g^{-1})
    = \frac{1}{\card (G)} \sum_{g \in G} \tr (a)
    = \tr (a).
\]
This is~(\ref{L_2Y15_IrrEq_Tr}).

Finally, assume~(\ref{L_2Y15_IrrEq_Tr}).
We prove~(\ref{L_2Y15_IrrEq_Cnt}).
Suppose $a \in M_d$ commutes with every element of~$G.$
Then
\[
a = \frac{1}{\card (G)} \sum_{g \in G} g a g^{-1}
  = \tr (a) \cdot 1,
\]
so is a scalar multiple of~$1.$
\end{proof}

\begin{cor}\label{C_2Z22_ProdIrr}
Let $d_1, d_2 \in \N$
and let $G_1 \subset \inv (M_{d_1})$ and $G_2 \subset \inv (M_{d_2})$
be finite subgroups.
Set
\[
G = \big\{ g_1 \otimes g_2 \colon
      {\mbox{$g_1 \in G_1$ and $g_2 \in G_2$}} \big\}
   \subset \inv (M_{d_1} \otimes M_{d_2})
   = \inv (M_{d_1 d_2}).
\]
Then $G$ is a finite subgroup of $\inv (M_{d_1 d_2}),$
and
if the natural actions of $G_1$ on $\C^{d_1}$
and $G_2$ on $\C^{d_2}$ are irreducible,
then the natural action of $G$ on $\C^{d_2 d_2}$ is irreducible.
\end{cor}

\begin{proof}
The first part is obvious.
For irreducibility,
we use Lemma~\ref{L_2Y15_IrrEq}(\ref{L_2Y15_IrrEq_Scl}),
and it is enough to check the condition on elements of the
form $a \otimes b.$
We have
\[
\sum_{g \in G} g (a \otimes b) g^{-1}
% = \sum_{g_1 \in G_1} \sum_{g_2 \in G_2}
%          g_1 a g_1^{-1} \otimes g_2 b g_2^{-1}
  = \left( \sum_{g_1 \in G_1} g_1 a g_1^{-1} \right)
         \otimes \left( \sum_{g_2 \in G_2} g_2 b g_2^{-1} \right),
\]
which is the tensor product of scalars and so a scalar.
\end{proof}

For those readers only concerned with simplicity of $\OP{d}{p},$
the first of the following examples is the only one that is needed.
Before stating it, we introduce the finite subgroup
of $\inv (M_d)$ which we most commonly use.

\begin{dfn}\label{D_2Y16_SgnPerm}
Let $d \in \N.$
Let $(e_{j, k})_{j, k = 1}^d$ be the standard system
of matrix units in $M_d,$
and let $S_d$ be the symmetric group.
The subgroup of $\inv (M_d)$ given by
\[
G = \left\{ \sum_{j = 1}^d \ep_j e_{j, \sm (j)} \colon
      {\mbox{$\sm \in S_d$
         and $\ep_1, \ep_2, \ldots, \ep_d \in \{ -1, 1 \}$}} \right\}.
\]
is called the group of
{\emph{signed permutation matrices}}.
\end{dfn}

\begin{lem}\label{L_2Y16_PrpSgPm}
Let $d \in \N$ and let $p \in [1, \I].$
Then the signed permutation matrices act on~$l_d^p$
via isometries,
and this action is irreducible.
Moreover, the signed permutation matrices span~$M_d.$
\end{lem}

\begin{proof}
It is immediate that the signed permutation matrices act on~$l_d^p$
via isometries.
Irreducibility follows easily from the last sentence.
For the last sentence, let $G$ be the group of signed permutation matrices.
% It suffices to show that $\spn (G) = M_d.$
Let the notation be as in Definition~\ref{D_2Y16_SgnPerm}.
It suffices to show that $e_{j, k} \in \spn (G)$
for all $j, k \in \{ 1, 2, \ldots, d \}.$
Choose $\sm\in S_d$
such that $\sm (j) = k.$
Then
\[
g = \sum_{l = 1}^d e_{l, \sm (l)}
\andeqn
h = e_{j, k} - \sum_{l \neq j} e_{l, \sm (l)}
\]
satisfy $g, h \in G$ and $e_{j, k} = \tfrac{1}{2} (g + h).$
So $e_{j, k} \in \spn (G).$
\end{proof}

\begin{exa}\label{E_2Y15_MdIdRep}
Let $p \in [1, \I)$
and let $d \in \N.$
Then the identity \rpn{} of $M_d$ on $l_d^p$
has enough isometries.
This is immediate from Lemma~\ref{L_2Y16_PrpSgPm}.
\end{exa}

\begin{exa}\label{E_2Y15_Sim}
Let $\XBM$ be a \sfm,
let $p \in [1, \I),$
and let $d \in \N.$
Let $\rh, \sm \colon M_d \to \LLp$
be unital representations,
and suppose that there is $u \in \inv (M_d)$
such that $\sm (a) = \rh (u a u^{-1})$
for all $a \in M_d.$
If $\rh$ has enough isometries,
then so does~$\sm.$

To see this, let $G \subset \inv (A)$ be a finite subgroup
witnessing the fact that $\rh$ has enough isometries.
Then use $u^{-1} G u$ to show that $\sm$ has enough isometries.
\end{exa}

\begin{exa}\label{E_2Y15_pSum}
Let $p \in [1, \I)$
and let $d \in \N.$
Let $I$ be a countable index set,
and for $i \in I$ let $(X_i, {\mathcal{B}}_i, \mu_i)$
be a \sfm{}
and let $\rh_i \colon M_d \to L (L^p (X_i, \mu_i))$
be a \rpn{} which locally has enough isometries.
Assume that $\sup_{i \in I} \| \rh_i (a) \| < \I$ for all $a \in M_d.$
Set $X = \coprod_{i \in I} X_i,$
with the obvious measure~$\mu.$
Then the $p$-direct sum
(see Remark~2.15 of~\cite{Ph5} and the discussion afterwards)
$\rh = \bigoplus_{i \in I} \rh_i \colon M_d \to \LLp$
locally has enough isometries.
\end{exa}

\begin{exa}\label{E_2Y15_CptEn}
Let $p \in [1, \I)$
and let $d \in \N.$
% Give $M_d$ the norm it carries when identified with $\MP{d}{p}.$
Let $K \subset \inv (M_d)$
be a nonempty compact set.
Then there exist a probability space $\XBM$
and a unital representation $\rh \colon \MP{d}{p} \to \LLp$
which locally has enough isometries
and such that for all $a \in \MP{d}{p}$ we have
$\| \rh (a) \| = \sup_{u \in K} \| u a u^{-1} \|.$
If $1 \in K,$
we can arrange that $\rh$ dominates the spatial representation.

Choose $u_1, u_2, \ldots \in K$
such that $\{ u_1, u_2, \ldots \}$ is dense in~$K.$
If $1 \in K,$
we require that $u_1 = 1.$
Take $X = \{ 1, 2, \ldots, d \} \times \N$
with the atomic measure $\mu$ determined by
$\mu ( \{ (j, n) \} ) = 2^{-n} d^{-1}$ for $j = 1, 2, \ldots, d$
and $n \in \N.$
For $n \in \N,$
take $X_n = \{ 1, 2, \ldots, d \} \times \{ n \}$
and let $\mu_n$ be the restriction of $\mu$ to~$X_n.$
Identify $L^p (X_n, \mu_n)$ with $L (l_d^p)$ in the obvious way
(incorporating the factor $2^{- n} d^{-1}$
by which the measures differ).
Let $\sm_n \colon M_d \to L (L^p (X_n, \mu_n))$
be the identity representation gotten from this identification.
Define a \rpn{} $\rh_n \colon M_d \to L (L^p (X_n, \mu_n))$
by $\rh_n (a) = \sm_n (u_n a u_n^{-1})$ for $a \in L_d.$
Then let $\rh \colon M_d \to \LLp$
be the $p$-direct sum of the \rpn{s}~$\rh_n$
as in Example~\ref{E_2Y15_pSum}.
The condition $\sup_{n \in \N} \| \rh_n (a) \| < \I$
is satisfied because $K$ is compact.
Using density of $\{ u_1, u_2, \ldots \}$ at the second step,
we get
\[
\| \rh (a) \|
 = \sup_{n \in \N} \| u_n a u_n^{-1} \|
 = \sup_{u \in K} \| u a u^{-1} \|.
\]
The \rpn{} $\rh$ locally has enough isometries
by Example~\ref{E_2Y15_pSum}.
\end{exa}

\begin{lem}\label{L_2Y15_TPrd}
Let $p \in [1, \I)$
and let $m, n \in \N.$
Identify $\MP{m n}{p}$ with $\MP{m}{p} \otimes_p \MP{n}{p}$
as in Corollary~\ref{C_2Y21MpTens}.
Let $\XBM$ and $\YCN$ be \sfm{s},
and let $\rh \colon \MP{m}{p} \to \LLp$
and $\sm \colon \MP{n}{p} \to L (L^p (Y, \nu))$
be unital \rpn{s}.
Let $\rh \otimes_p \sm \colon
 \MP{m n}{p} \to L \big( L^p (X \times Y, \, \mu \times \nu) \big)$
be as in Lemma~\ref{L_2Z29_Lower}.
\begin{enumerate}
\item\label{L_2Y15_TPrd_EnI}
If $\rh$ and $\sm$ have enough isometries,
then so does
$\rh \otimes_p \sm.$
\item\label{L_2Y15_TPrd_LEI}
If $\rh$ and $\sm$ locally have enough isometries,
then so does
$\rh \otimes_p \sm.$
\item\label{L_2Y15_TPrd_DmSp}
If $\rh$ and $\sm$ dominate the spatial representation,
then so does
$\rh \otimes_p \sm.$
\end{enumerate}
\end{lem}

\begin{proof}
We prove~(\ref{L_2Y15_TPrd_EnI}).
Let $G \subset \inv (\MP{m}{p})$ and $H \subset \inv ( \MP{n}{p} )$
be finite subgroups witnessing the fact that
$\rh$ and $\sm$ have enough isometries.
Define
% $K \subset \inv (\MP{m}{p} \otimes_p \MP{n}{p})$ to be
\[
K = \big\{ g \otimes h \colon
      {\mbox{$g \in G$ and $h \in H$}} \big\}
   \subset \inv (\MP{m}{p} \otimes_p \MP{n}{p})
   = \inv (\MP{m n}{p}).
\]
Then $K$ is a finite subgroup.
Since $K$ is a group,
Theorem~\ref{T-LpTP}(\ref{T-LpTP-3b})
implies that $(\rh \otimes_p \sm) (k)$
is an isometry for all $k \in K.$
To finish the proof,
we need only check that the natural action of $K$ on $\C^{m n}$
is irreducible.
This follows from Corollary~\ref{C_2Z22_ProdIrr}.

For~(\ref{L_2Y15_TPrd_LEI}),
if $X = \coprod_{i \in I} X_{i}$
and $Y = \coprod_{j \in J} Y_{j}$
are decompositions for $\rh$ and $\sm$
as in Definition~\ref{D_2Y15_EnId}(\ref{D_2Y15_EnId_Loc}),
then, using~(\ref{L_2Y15_TPrd_EnI}),
the decomposition
$X \times Y = \coprod_{i \in I} \coprod_{j \in J} X_i \times Y_j$
shows that $\rh \otimes_p \sm$ locally has enough isometries.

We prove~(\ref{L_2Y15_TPrd_DmSp}).
Continue with the notation in the proof of~(\ref{L_2Y15_TPrd_LEI}),
and let $\mu$ and $\nu$ be the measures on $X$ and~$Y.$
By hypothesis, there are $i_0 \in I$ and $j_0 \in J$
such that the
\rpn{s} $x \mapsto \rh (x) |_{L^p (X_{i_0}, \mu)}$
and $y \mapsto \sm (y) |_{L^p (Y_{j_0}, \nu)}$
are spatial.
Let $Z$ be the component $X_{i_0} \times Y_{j_0}$
of the decomposition in the proof of~(\ref{L_2Y15_TPrd_LEI}).
Then $z \mapsto (\rh \otimes_p \sm) (z) |_{L^p (Z, \mu \times \nu)}$
is spatial by Lemma~\ref{L_2Y21MpTRep}.
\end{proof}

\begin{dfn}\label{D_2Y15_IdPr}
Let $B$ be a unital Banach algebra,
let $A \subset B$ be a unital subalgebra,
and let $E \colon B \to A$ be a Banach conditional expectation
(\Def{D:CondExpt}).
We say that $E$ is
{\emph{ideal preserving}}
if for every closed ideal $I \subset B,$
we have $E (I) \subset I \cap A.$
\end{dfn}

Banach conditional expectations
are usually not ideal preserving.
For instance, those in Example~\ref{E_2Y16_HmIsE}
are almost never ideal preserving.

\begin{lem}\label{L:CondExptMd}
Let $d \in \N$ and let $p \in [1, \I).$
Let $\XBM$ and $\YCN$ be \sfm{s}.
Let $\rh \colon M_d \to \LLp$
be a representation which locally has enough isometries
in the sense of Definition~\ref{D_2Y15_EnId}(\ref{D_2Y15_EnId_Loc}),
and set $D = \rh (M_d).$
Let $A$ be any closed unital subalgebra of $L (L^p (Y, \nu)).$
Let $D \otimes_p A$
be as in \Def{D:LpTensorPrd}.
Identify $A$ with $\{ 1 \otimes a \colon a \in A \}.$
Then there exists a unique
continuous linear map $E \colon D \otimes_p A \to A$
such that
$E (\rh (y) a) = \tr (y) a$ for all $a \in A$ and $y \in M_d,$
and $E$ is an ideal preserving Banach conditional expectation.
\end{lem}

\begin{proof}
Uniqueness follows from the fact that the elementary tensors
span a dense subset of $D \otimes_p A.$
(See \Def{D:LpTensorPrd}.)

Now let $G \subset \inv (M_d)$ be any finite subgroup
whose natural action on $\C^d$ is irreducible.
For example, see Definition~\ref{D_2Y16_SgnPerm}.
Define $E \colon D \otimes_p A \to D \otimes_p A$
by
\begin{equation}\label{Eq_3715_Star}
E (x) = \frac{1}{\card (G)} \sum_{g \in G}
    (\rh (g) \otimes 1) x (\rh (g) \otimes 1)^{-1}
\end{equation}
for $x \in D \otimes_p A.$
Then $E (\rh (y) a) = \tr (y) a$ for all $a \in A$ and $y \in M_d$
by Lemma~\ref{L_2Y15_IrrEq}(\ref{L_2Y15_IrrEq_Tr}).
Therefore the range of $E$ is contained in~$A.$
Now Remark~\ref{R_2Y16_AvgBddAct}
shows that $E$
has all the properties of a Banach conditional expectation,
except that we do not yet know that $\| E \| = 1.$
This proves existence except for the norm estimate.
Moreover,
$E$ does not depend on the choice of the finite group~$G.$

It remains to estimate $\| E \|.$
Let $X = \coprod_{i \in I} X_{i}$
be a partition of $X$
as in \Def{D_2Y15_EnId}(\ref{D_2Y15_EnId_Loc}),
and let $G_i \subset \inv (M_d)$ be a finite subgroup
whose natural action on $\C^d$ is irreducible
and such that $\rh (g) |_{L^p (X_i, \mu)}$
is isometric for all $g \in G_i.$
Then for every $i \in I,$
the space $L^p (X_i \times Y, \, \mu \times \nu)$
is invariant under the action of $D \otimes_p A.$
Since $L^p (X \times Y, \, \mu \times \nu)$ is
the $L^p$~sum of these subspaces,
it suffices to show that
for every $i \in I$ and $\xi \in L^p (X_i, \mu),$
we have $\| E (x) \xi \| \leq \| x \| \cdot \| \xi \|$
for all $x \in D \otimes_p A.$
Using the fact that $E$ in~(\ref{Eq_3715_Star})
is independent of the choice of~$G$
at the first step
and the fact that $\rh (g) |_{L^p (X_i, \mu)}$is isometric
at the third step,
we have
\begin{align*}
\| E (x) \xi \|
& = \left\| \frac{1}{\card (G_i)} \sum_{g \in G_i}
    (\rh (g) \otimes 1) x (\rh (g) \otimes 1)^{-1} \xi \right\|
       \\
& \leq \frac{1}{\card (G_i)} \sum_{g \in G_i}
    \| (\rh (g) \otimes 1) x (\rh (g) \otimes 1)^{-1} \xi \|
       \\
& \leq \frac{1}{\card (G_i)} \sum_{g \in G_i} \| x \| \cdot \| \xi \|
  = \| x \| \cdot \| \xi \|,
\end{align*}
as desired.
\end{proof}

The following result gives the reason we want
our Banach conditional expectations to be ideal preserving.

\begin{lem}\label{L_2Y15_IdPLim}
Let $A$ be a unital Banach algebra.
Let $A_0 \subset A_1 \subset \cdots \subset A$
be closed unital subalgebras
such that ${\overline{\bigcup_{n = 0}^{\I} A_n}} = A$
and such that there are ideal preserving Banach conditional expectations
$E_n \colon A \to A_n.$
Let $I \subset A$ be a closed ideal.
Then $I = {\overline{\bigcup_{n = 0}^{\I} (I \cap A_n)}}.$
\end{lem}

\begin{proof}
We first claim that for all $a \in A$ we have $\limi{n} E_n (a) = a.$
So let $a \in A$ and let $\ep > 0.$
Choose $N \in \Nz$ and $x \in A_N$ such that
$\| a - x \| < \frac{\ep}{2}.$
Let $n \geq N.$
Then $E_n (x) = x,$
so
\[
\| E_n (a) - a \|
  \leq \| E_n (a - x) \| + \| x - a \|
  \leq \| E_n \| \cdot \| a - x \| + \| x - a \|
  < \tfrac{\ep}{2} + \tfrac{\ep}{2}
  = \ep.
\]
The claim follows.

Now let $a \in I.$
Then for all $n \in \Nz$ we have $E_n (a) \in I \cap A_n,$
and $\limi{n} E_n (a) = a,$
so $a \in {\overline{\bigcup_{n = 0}^{\I} (I \cap A_n)}}.$
\end{proof}

The following proposition gives the properties of the eigenspaces
of an action of a compact abelian group on a Banach algebra.
All these properties are well known for an action on a \ca,
and one really only needs to check that nothing goes wrong
for Banach algebras.

\begin{prp}\label{P:GpEigensp}
Let $G$ be a compact abelian group,
with normalized Haar measure~$\nu$ and Pontryagin dual~${\widehat{G}}.$
Let $B$ be a Banach algebra,
and let $\bt \colon G \to \Aut (B)$ be an action of $G$ on~$B.$
For $\ta \in {\widehat{G}}$ define
\[
B_{\ta} =
 \big\{ b \in B \colon {\mbox{$\bt_g (b) = \ta (g) b$ for all
                        $g \in G$}} \big\}.
\]
Take Banach space valued integration to be as in Proposition~\ref{P:Avg}
or the remark afterwards,
and define $P_{\ta} \colon B \to B$ by
\[
P_{\ta} (b) = \int_G {\ov{\ta (g)}} \bt_g (b) \, d \nu (g)
\]
for $b \in B.$
Then:
\begin{enumerate}
\item\label{P:GpEigensp-1}
The number $M = \sup_{g \in G} \| \bt_g \|$ is finite.
\item\label{P:GpEigensp-2}
With $M$ as in~(\ref{P:GpEigensp-1}),
we have
$\| P_{\ta} \| \leq M$ for all $\ta \in {\widehat{G}}.$
\item\label{P:GpEigensp-3}
${\operatorname{ran}} (P_{\ta}) = B_{\ta}$
for all $\ta \in {\widehat{G}}.$
\item\label{P:GpEigensp-4a}
$P_{\ta} \circ P_{\ta} = P_{\ta}$ for all $\ta \in {\widehat{G}}.$
\item\label{P:GpEigensp-4b}
$P_{\sm} \circ P_{\ta} = 0$ for all $\sm, \ta \in {\widehat{G}}$
with $\sm \neq \ta.$
\item\label{P:GpEigensp-5}
$B_{\ta}$ is a closed subspace of~$B$ for all $\ta \in {\widehat{G}}.$
\item\label{P:GpEigensp-6}
$B_{\sm} B_{\ta} \subset B_{\sm \ta}$
for all $\sm, \ta \in {\widehat{G}}.$
\item\label{P:GpEigensp-7}
For all $\sm, \ta \in {\widehat{G}},$
$c \in B_{\sm},$ and $b \in B,$
we have
$P_{\ta} (b c) = P_{\ta \sm^{-1}} (b) c$
and $P_{\ta} (c b) = c P_{\ta \sm^{-1}} (b).$
\item\label{P:GpEigensp-9}
If $b \in B$ satisfies
$P_{\ta} (b) = 0$ for all $\ta \in {\widehat{G}},$
then $b = 0.$
% \item\label{P:GpEigensp-10}
% \item\label{P:GpEigensp-11}
% \item\label{P:GpEigensp-12}
\end{enumerate}
\end{prp}
\begin{proof}
Part~(\ref{P:GpEigensp-1}) is immediate from the
Uniform Boundedness Principle.
Part~(\ref{P:GpEigensp-2}) is immediate from~(\ref{P:GpEigensp-1}).

Next, we observe that
for $\ta \in {\widehat{G}},$ $g \in G,$ and $b \in B,$
we have
\begin{equation}\label{Eq:gPtau}
\bt_g (P_{\ta} (b))
  = \int_G {\ov{\ta (h)}} \bt_{g h} (b) \, d \nu (h)
  = \int_G {\ov{\ta (g^{-1} h)}} \bt_{h} (b) \, d \nu (h)
  = \ta (g) P_{\ta} (b).
\end{equation}
So ${\operatorname{ran}} (P_{\ta}) \subset B_{\ta}.$
For $\ta \in {\widehat{G}}$ and $b \in B_{\ta},$
one immediately checks that $P_{\ta} (a) = a.$
Therefore ${\operatorname{ran}} (P_{\ta}) = B_{\ta},$
which is~(\ref{P:GpEigensp-3}),
and $P_{\ta} \circ P_{\ta} = P_{\ta},$
which is~(\ref{P:GpEigensp-4a}).
Moreover,
if $\sm \in {\widehat{G}},$
then (\ref{Eq:gPtau}) implies that
\[
P_{\sm} (P_{\ta} (b))
  = \int_G {\ov{\sm (h)}} \bt_{h} ( P_{\ta} (b)) \, d \nu (h)
  = \left( \int_G {\ov{\sm (h)}} \ta (h) \, d \nu (h) \right) P_{\ta} (b),
\]
which is zero if $\sm \neq \ta.$
This is~(\ref{P:GpEigensp-4b}).

Part~(\ref{P:GpEigensp-5}) is immediate from continuity
of the automorphisms $\bt_g,$
and Part~(\ref{P:GpEigensp-6}) is immediate from the fact that
each $\bt_g$ is a \hm.

For part~(\ref{P:GpEigensp-7}), we compute:
\begin{align*}
P_{\ta} (b c)
& = \int_G {\ov{\ta (g)}} \bt_g (b c) \, d \nu (g)
         \\
& = \int_G {\ov{\ta (g)}} \bt_g (b) \sm (g) c \, d \nu (g)
  = \left( \int_G
      {\ov{\ta (g) \sm^{-1} (g)}} \bt_g (b) \, d \nu (g) \right) c
  = P_{\ta \sm^{-1}} (b) c.
\end{align*}
The part involving $P_{\ta} (c b)$ is done similarly.

Finally, we prove Part~(\ref{P:GpEigensp-9}).
Suppose $P_{\ta} (b) = 0$ for all $\ta \in {\widehat{G}}.$
Let $B'$ be the Banach space dual of~$B.$
Let $\om \in B'$ be arbitrary.
Then for all $\ta \in {\widehat{G}}$ we have
\[
0 = \om ( P_{\ta} (b))
  = \int_G {\ov{\ta (g)}} \om (\bt_g (b)) \, d \nu (g).
\]
So $g \mapsto f_{\om} (g) = \om (\bt_g (b))$ is a \cfn{} on $G$
whose Fourier transform ${\widehat{f_{\om}}}$
vanishes on ${\widehat{G}}.$
Therefore $f_{\om} = 0.$
Since this is true for all $\om \in B',$
the Hahn-Banach Theorem implies that $\bt_g (b) = 0$
for all $g \in G.$
In particular, $b = 0,$
as desired.
\end{proof}

\section{$L^p$~UHF algebras}\label{Sec_N_UHF}

\indent
We define an $L^p$~version of the well known
family of UHF \ca s.
We need these algebras and some of their properties
to prove simplicity of the $L^p$~Cuntz algebras,
just as certain UHF \ca s were used in the proof~\cite{Cu1}
that the \ca{} ${\mathcal{O}}_d$ is simple.
We actually only need simplicity of the spatial $L^p$~UHF algebras
(Definition~\ref{D_2Y15_ITP} below).
However,
with some preparation,
the same argument gives simplicity for a much larger class.

The key step in the proof of simplicity is the part of the proof
of Lemma~\ref{L:UHFCondExpt}
which shows that $E_{R, S}$ is ideal preserving
when $S \setminus R$ is not finite.

When $p = 2,$
a spatial $L^p$~UHF algebra is just a UHF \ca,
as originally introduced in~\cite{Gl}.

In this section,
we prove the following results about $L^p$~UHF algebras:
\begin{itemize}
\item
For every supernatural number~$N$
(Definition~\ref{D_2Y15_SNat}) and every $p \in [1, \I),$
there is a spatial $L^p$~UHF algebra of type~$N,$
and it can be represented isometrically on $L^p (X, \mu)$
for some \sfm{} $\XBM.$
(See Theorem~\ref{T_2Y15_SpEx}.)
\item
Every unital contractive \hm{} from a spatial $L^p$~UHF algebra
to $\LLp,$
for any \sfm{} $\XBM,$ is isometric.
(See Theorem \ref{T_2Y15_SpEx}(\ref{T_2Y15_SpEx_Contr}).)
\item
For a fixed supernatural number~$N$ and fixed $p \in [1, \I),$
any two spatial $L^p$~UHF algebras of type~$N$
are isometrically isomorphic.
(See Theorem~\ref{T:UHFIsoCrit},
taking $D = \C.$)
\item
For every $p \in [1, \I),$
any $L^p$~UHF algebra of tensor product type
which locally has enough isometries
(Definition~\ref{D_2Y15_ITP}) is simple (Theorem~\ref{T:LpUHFSimple})
and has a unique normalized trace (Corollary~\ref{C:LpUHFTrace}).
(This class includes the spatial $L^p$~UHF algebras.)
\item
Every spatial $L^p$~UHF algebra is an amenable Banach algebra
in the sense of Definition 2.1.9 of~\cite{Rnd}.
(See Theorem~\ref{T_3714Amen}.)
\end{itemize}

In~\cite{Ph-Lp2b},
we will show that general $L^p$~UHF algebras need not be amenable.

The result we need later is simplicity.
However,
the other results are easy to get,
so it seems appropriate to include them.

We will need direct systems of Banach algebras with contractive
\hm s indexed by~$\Nz.$

\begin{dfn}\label{D:DirSys}
A {\emph{contractive direct system of Banach algebras}}
is a pair
\[
\big( (A_n)_{n \in \Nz}, \, ( \ph_{n, m} )_{m \leq n} \big)
\]
consisting of a family $(A_n)_{n \in \Nz}$ of Banach algebras
and a family $( \ph_{n, m} )_{m \leq n}$ of contractive
\hm s $\ph_{n, m} \colon A_m \to A_n$
for $m, n \in \Nz$ with $m \leq n,$
such that $\ph_{n, n} = \id_{A_n}$ for all $n \in \Nz$
and $\ph_{n, m} \circ \ph_{m, l} = \ph_{n, l}$
whenever $l \leq m \leq n.$

The {\emph{direct limit}} $\dirlim A_n$ of this direct system is the
Banach algebra direct (``inductive'') limit,
as constructed in much greater generality
in Section~3.3 of~\cite{Bl3}.
\end{dfn}

\begin{rmk}\label{R-DlimUP}
Let the notation be as in \Def{D:DirSys},
and let $A = \dirlim A_n.$
Then $A$ has the usual universal property.
It is not stated in~\cite{Bl3}, but is easily proved;
the equivariant
\ca{} proof in Proposition 2.5.1 of~\cite{Ph0} is very similar.
Here, since our maps are contractive,
it takes the following form.
Suppose we are given a Banach algebra~$B$ 
and contractive \hm s $\ps_n \colon A_n \to B$
for $n \in \Nz$
such that $\ps_n \circ \ph_{n, m} = \ps_m$
for $m, n \in \Nz$ with $m \leq n.$
Then there exists a unique contractive \hm{} $\ps \colon A \to B$
such that $\ps \circ \ph_n = \ps_n$ for all $n \in \Nz.$
\end{rmk}

We only need the situation in which all the maps $\ph_{n, m}$
are isometric,
in which case it is obvious that we in fact get a norm on
the algebraic direct limit,
and that the Banach algebra direct limit
is just the completion in this norm.

If $A$ is a Banach algebra and
$A_0 \subset A_1 \subset A_2 \subset \cdots$
is a sequence of closed subalgebras
such that $A = {\overline{\bigcup_{n = 0}^{\I} A_n}},$
and for $m, n \in \Nz$ with $m \leq n$
we take $\ph_{n, m} \colon A_m \to A_n$ to be the inclusion,
it is clear that $A = \dirlim A_n.$

Part of the following definition
is contained in Definition~1.3 of~\cite{Gl}.

\begin{dfn}\label{D_2Y15_SNat}
Let $P$ be the set of prime numbers.
A {\emph{supernatural number}} is a function
$N \colon P \to \N \cup \{ \infty \}$
such that $\sum_{t \in P} N (t) = \infty.$

Let $d = (d (1), \, d (2), \, \ldots )$
be a sequence of integers such that $d (n) \geq 2$ for all $n \in \N.$
We define a new sequence
$r_d$ by
\[
r_d (n) = d (1) d (2) \cdots d (n)
\]
for $n \in \Nz.$
(Thus, $r_d (0) = 1.$)
We define a function
$N_d \colon P \to \Nz \cup \{ \I \}$
by
\[
N_d (t) = \sup \big( \big\{ k \in \Nz \colon
        {\mbox{there is $n \in \Nz$ such that $t^k$ divides $r_d (n)$}}
              \big\} \big).
\]
The function $N_d$ is called the {\emph{supernatural number
associated with~$d$}}.

If $m \in \{ 2, 3, 4, \ldots \},$ and if $N (t) = \infty$
for those primes $t$ which divide~$m$ and $N (t) = 0$ for all
other primes~$t,$
we write $m^{\I}$ for~$N.$
\end{dfn}

\begin{dfn}\label{D_2Y14_pUHFTypeN}
Let $\XBM$ be a \sfm,
let $p \in [1, \infty),$
and let $A \subset \LLp$ be a unital subalgebra.
Let $N$ be a supernatural number.
We say that $A$ is an
{\emph{$L^p$~UHF algebra of type~$N$}}
if there exist a sequence $d$ as in Definition~\ref{D_2Y15_SNat}
with $N_d = N,$
unital subalgebras $D_0 \subset D_1 \subset \cdots \subset A,$
and,
with $r_d$ as in Definition~\ref{D_2Y15_SNat},
algebraic isomorphisms $\sm_n \colon M_{r_d (n)} \to D_n,$
such that $A = {\overline{\bigcup_{n = 0}^{\I} D_n}}.$
We say that $A$ {\emph{dominates the spatial norm}}
if,
using the norm from $\MP{r_d (n)}{p},$
the algebraic isomorphisms $\sm_n$ above
can be chosen such that:
\begin{enumerate}
\item\label{D_2Y14_pUHFTypeN_Ineq}
$\| \sm_n (a) \| \geq \| a \|$
for all $a \in M_{r_d (n)}.$
\item\label{D_2Y14_pUHFTypeN_Cons}
Identifying $M_{r_d (n + 1)}$ with $M_{r_d (n)} \otimes M_{d (n + 1)},$
we have $\sm_{n + 1} (a \otimes 1) = \sm_n (a)$
for all $a \in M_{r_d (n)}.$
\end{enumerate}
\end{dfn}

Essentially,
an $L^p$~UHF algebra is a direct limit
of matrix algebras with unital maps which is isometrically
isomorphic to an algebra of operators on some space
of the form $L^p (X, \mu).$

We think of domination of the spatial norm
as a normalization.
In Example~\ref{E_2Y14_LpUHF} below,
in which we describe
$L^p$~UHF algebras of tensor product type,
the obvious source of the representations of the tensor
factors $M_{d (n)}$ on $L^p$-spaces
is Example~\ref{E_2Y15_CptEn},
in which for $a \in M_n$
we got $\| \rh (a) \| = \sup_{u \in K} \| u a u^{-1} \|$
for a fixed compact set $K \subset \inv (M_d).$
For domination of the spatial norm,
it suffices to require $1 \in K.$

\begin{dfn}\label{D:StdUHF}
Let $p \in [1, \infty)$
and let $N$ be a supernatural number.
A
{\emph{spatial direct system of type~$N$}}
is any direct system
$\big( (A_n)_{n \in \Nz}, \, ( \ph_{n, m} )_{m \leq n} \big)$
of Banach algebras such that
the maps $\ph_{n, m}$ are all isometric and unital,
and for which
there exist a sequence $d$ as in Definition~\ref{D_2Y15_SNat}
with $N_d = N$
and isometric isomorphisms $A_n \cong \MP{r_d (n)}{p}.$
We can always assume that $A_0 = \C.$
\end{dfn}

The term ``spatial'' is used because, for $p \neq 2,$
the condition on the maps
$\MP{r_d (n - 1)}{p} \to \MP{r_d (n)}{p}$
implicit in the definition is exactly that they are spatial.
See Theorem~\ref{T_2Y17Iso}.
For $p = 2,$
Lemma~\ref{L_3717_CSt} implies that a spatial direct system
is a direct system of \ca{s} and *-\hm{s}.

The direct limit of a spatial direct system of type~$N$
is in fact an $L^p$~UHF algebra of type~$N,$
but at the moment it is not obvious
that such a direct limit can be represented as an algebra of operators
on an $L^p$~space.

We can use the usual intertwining argument for \ca s
(see, for example, the proof of Theorem~4.3 of~\cite{Ell1})
to prove that the direct limits of any two spatial
direct systems of the same type are isometrically isomorphic.
We need a lemma.

\begin{lem}\label{L:IsosAreEq}
Let $r, s \in \N,$ and suppose $r$ divides~$s.$
Let $p \in [1, \I).$
Then:
\begin{enumerate}
\item\label{L:IsosAreEq-Ex}
There is a unital isometric \hm{}
$\rh \colon \MP{r}{p} \to \MP{s}{p}.$
\item\label{L:IsosAreEq-Uq}
Let $\ph, \ps \colon \MP{r}{p} \to \MP{s}{p}$
be unital isometric \hm s.
Then there exists an isometry $u \in \MP{s}{p}$
such that $\ps (a) = u \ph (a) u^{-1}$ for all $a \in \MP{r}{p}.$
\end{enumerate}
\end{lem}

\begin{proof}
The result is well known for $p = 2,$
so assume $p \neq 2.$
To construct~$\rh$ in~(\ref{L:IsosAreEq-Ex}),
let
\[
X = \{ 1, 2, \ldots, r \}
\andeqn
Y = \{ 1, 2, \ldots, s / r \}.
\]
Then $l_r^p = l^p (X)$
and $l_s^p$ is isometrically isomorphic to $l^p (X \times Y).$
We take $\rh$ to be \hm{}
\[
\rh \colon L (l^p (X)) \to L \big( l^p (X \times Y ) \big)
\]
given by $\rh (a) = a \otimes 1.$
This \hm{} is isometric by
Theorem~\ref{T-LpTP}(\ref{T-LpTP-3b}).

To prove~(\ref{L:IsosAreEq-Uq}), it suffices to take $\ph = \rh.$
The existence of~$u$ then follows from
(1) implies~(8)
of Theorem~7.2 of~\cite{Ph5}.
% (\ref{T:SpatialRepsMd-1}) implies~(\ref{T:SpatialRepsMd-8})
% of Theorem~\ref{T:SpatialRepsMd}.
\end{proof}

In the following theorem,
we prove uniqueness not only for the spatial $L^p$~UHF algebra
of type~$N$
(as we will later call the algebra $A$ in the theorem)
but also for its tensor product with
a unital subalgebra $D \subset L (L^p (Y, \nu)).$
The extra generality will be used elsewhere.
One can eliminate most of the first part of the proof,
and get the result needed for this paper,
by taking $D = \C.$

\begin{thm}\label{T:UHFIsoCrit}
Let $p \in [1, \infty)$
and let $N$ be a supernatural number.
Let
\[
\big( (A_n)_{n \in \Nz}, \, ( \ph_{n, m} )_{m \leq n} \big)
\andeqn
\big( (B_n)_{n \in \Nz}, \, ( \ps_{n, m} )_{m \leq n} \big)
\]
be spatial direct systems of type~$N,$
as in Definition~\ref{D:StdUHF},
with $A_0 = B_0 = \C$
and with direct limits $A = \dirlim A_n$ and $B = \dirlim B_n.$
Let $\YCN$ be a \sfm,
and let $D \subset L (L^p (Y, \nu))$ be a unital subalgebra.
Then:
\begin{enumerate}
\item\label{T:UHFIsoCrit_DS}
There are direct systems with isometric maps
\[
\big( (A_n \otimes_p D)_{n \in \Nz}, \,
 ( \ph_{n, m} \otimes_p \id_D )_{m \leq n} \big)
\andeqn
\big( (B_n \otimes_p D)_{n \in \Nz}, \,
 ( \ps_{n, m} \otimes_p \id_D )_{m \leq n} \big).
\]
\item\label{T:UHFIsoCrit_FromAB}
The maps $\sm_n \colon A_n \to A_n \otimes_p D$
and $\ta_n \colon B_n \to B_n \otimes_p D,$
given by $x \mapsto x \otimes 1_D$
for $x \in A_n$ and $x \in B_n,$
combine to give isometric injections
$\sm \colon A \to \dirlim A_n \otimes_p D$
and $\ta \colon B \to \dirlim B_n \otimes_p D.$
\item\label{T:UHFIsoCrit_FD}
The maps $\ld_n \colon D \to A_n \otimes_p D$
and $\et_n \colon D \to B_n \otimes_p D,$
given by $y \mapsto 1 \otimes y$
for $y \in D,$
combine to give isometric injections
$\ld \colon D \to \dirlim A_n \otimes_p D$
and $\et \colon D \to \dirlim B_n \otimes_p D.$
\item\label{T:UHFIsoCrit_Iso}
There is an isometric isomorphism
$\gm \colon \dirlim A_n \otimes_p D \to \dirlim B_n \otimes_p D$
such that $\gm \circ \ld = \et$
and $\gm (\sm (A)) = \ta (B).$
\end{enumerate}
\end{thm}

\begin{proof}
By \Def{D:StdUHF},
there are sequences
\[
c = (c (1), \, c( 2), \, \ldots )
\andeqn
d = (d (1), \, d (2), \, \ldots )
\]
in $\{ 2, 3, \ldots \}$
such that $N_c = N_d = N,$
and
such that $A_n$ is isometrically isomorphic to $\MP{r_c (n)}{p}$
and $B_n$ is isometrically isomorphic to $\MP{r_d (n)}{p}$
for all $n \in \Nz.$

For $k \in \N,$
we define a norm on $\MP{k}{p} \otimes_{\text{alg}} D$
by representing it on $l_k^p \otimes_p L^p (Y, \nu)$
following Notation~\ref{N:FDP}
and Definition~\ref{D:LpTensorPrd}.
We write $\MP{k}{p} \otimes_p D$
for $\MP{k}{p} \otimes_{\text{alg}} D$ equipped with this norm.

Let $k \in \N.$
Set $N_k = \{ 1, 2, \ldots, k \},$
and let $\om$ be counting measure on~$N_k.$
Let $\XBM$ be a \sfm. 
We claim that a unital representation
$\rh \colon \MP{k}{p} \to \LLp$
is isometric if and only if
there exists a \sfm{} $(X_0, {\mathcal{B}}_0, \mu_0)$
and a bijective isometry
\[
u \colon L^p ( N_k \times X_0, \, \om \times \mu_0) \to L^p (X, \mu)
\]
such that,
identifying $L^p ( X_0 \times N_k, \, \mu_0 \times \om)$
with $L^p (X_0, \mu_0) \otimes_p l_k^p,$
for all $a \in \MP{k}{p}$ we have
$\rh (a) = u \big( 1_{L^p (X_0, \mu_0)}  \otimes a \big) u^{-1}.$
To prove the claim for $p \neq 2,$
use Theorem~\ref{T_2Y17Iso}
(part of Theorem~7.2 of~\cite{Ph5})
to see that $\rh$ is spatial,
and then use condition~(8) in Theorem~7.2 of~\cite{Ph5}.
For $p = 2,$
use Lemma~\ref{L_3717_CSt}
and the fact that any unital *-representation of $M_k$
on a Hilbert space is a multiple of the identity representation.

We now claim that for any \sfm{} $\XBM$
and any isometric representation
$\rh \colon \MP{k}{p} \to \LLp,$
the map
\[
\rh \otimes_p \id_D \colon
       \MP{k}{p} \otimes_p D
       \to \rh ( \MP{k}{p} ) \otimes_p D
       \subset L \big( L^p (X \times Y, \, \mu \times \nu) \big)
\]
is isometric.
Let $(X_0, {\mathcal{B}}_0, \mu_0)$ and $u$ be as in the previous claim.
Then
\[
(\rh \otimes_p \id_D) (x)
    = \big( u \otimes 1_{L^p (Y, \nu)} \big)
         \big( 1_{L^p (X_0, \mu_0)} \otimes x \big)
         \big( u^{-1} \otimes 1_{L^p (Y, \nu)} \big)
\]
for $x \in \MP{k}{p} \otimes_p D.$
It follows from Theorem \ref{T-LpTP}(\ref{T-LpTP-3b})
that $x \mapsto 1_{L^p (X_0, \mu_0)} \otimes x$
is isometric.
It also follows that
$\| u \otimes 1_{L^p (Y, \nu)} \|
 = \| u^{-1} \otimes 1_{L^p (Y, \nu)} \| = 1,$
so $u \otimes 1_{L^p (Y, \nu)}$ is a bijective isometry.
The claim follows.

This claim implies that
$A_n \otimes_p D \cong \MP{r_c (n)}{p} \otimes_p D$
and $B_n \otimes_p D \cong \MP{r_d (n)}{p} \otimes_p D$
isometrically for all $n \in \N.$
We may thus assume
$A_n = \MP{r_c (n)}{p}$ and $B_n = \MP{r_d (n)}{p}.$

The last claim also immediately implies that if $k_1, k_2 \in \N$
with $k_1 | k_2,$
and if $\ep \colon \MP{k_1}{p} \to \MP{k_2}{p}$
is unital and isometric,
then
\begin{equation}\label{Eq_4913_TpD}
\ep \otimes_p \id_D \colon
  \MP{k_1}{p} \otimes_p D \to \MP{k_2}{p} \otimes_p D
\end{equation}
is unital and isometric.
Part~(\ref{T:UHFIsoCrit_DS}) of the theorem now follows.

The maps $\sm_n,$ $\ta_n,$ $\ld_n,$ and $\et_n$
in (\ref{T:UHFIsoCrit_FromAB}) and~(\ref{T:UHFIsoCrit_FD})
are now all isometric by Theorem \ref{T-LpTP}(\ref{T-LpTP-3b}).
It is immediate that these maps are compatible with the direct systems,
and parts (\ref{T:UHFIsoCrit_FromAB}) and~(\ref{T:UHFIsoCrit_FD})
follow.

We now prove part~(\ref{T:UHFIsoCrit_Iso}).
We describe the intertwining argument
(from the proof of Theorem~2.14 of~\cite{Dxm})
for the benefit of the reader not familiar with it.

We inductively construct integers
\[
0 = m_0 < m_1 < \cdots
\andeqn
0 = n_0 < n_1 < \cdots
\]
and isometric unital \hm s
\[
\af_k \colon A_{m_k} \to B_{n_k}
\andeqn
\bt_k \colon B_{n_k} \to A_{m_{k + 1}}
\]
such that
\[
\bt_k \circ \af_k = \ph_{m_{k + 1}, m_k}
\andeqn
\af_{k + 1} \circ \bt_k = \ps_{n_{k + 1}, n_k}
\]
for $k \in \Nz,$
as shown in the following diagram:
\[
\xymatrix{
A_{m_0} \ar[r]^{\ph_{m_1, m_0}} \ar[d]_{\af_0}
  & A_{m_1} \ar[r]^{\ph_{m_2, m_1}} \ar[d]_{\af_1}
  & A_{m_2} \ar[r]^{\ph_{m_3, m_2}} \ar[d]_{\af_2}
  & \cdots \ar[r] & A \\
B_{n_0} \ar[r]_{\ps_{n_1, n_0}} \ar[ru]^{\bt_0}
  & B_{n_1} \ar[r]_{\ps_{n_2, n_1}} \ar[ru]^{\bt_1}
  & B_{n_2} \ar[r]_{\ps_{n_3, n_2}} \ar[ru]^{\bt_2}
  & \cdots \ar[r] & B,
}
\]
and such that
\[
\af_k \otimes_p \id_D \colon A_{m_k} \otimes_p D \to B_{n_k} \otimes_p D
\andeqn
\bt_k \otimes_p \id_D \colon
 B_{n_k} \otimes_p D \to A_{m_{k + 1}} \otimes_p D
\]
are also isometric.
We do not need to verify separately
that $\af_k \otimes_p \id_D$ and $\bt_k \otimes_p \id_D$
are isometric,
since we saw above that for $\ep$ isometric,
the map~(\ref{Eq_4913_TpD})
is always isometric.

Set $m_0 = n_0 = 0$ and set $\af_0 = \id_{\C}.$
Set $m_1 = 1$ and let $\bt_0 \colon B_{n_0} \to A_{m_1}$
be $\bt_0 (\ld) = \ld \cdot 1$ for $\ld \in \C.$

Since the exponent of each prime $t$ in the prime factorization of $m_1$
is no larger than $N_d (t),$
there is $n_1 > 0$ such that $r_c (m_1)$ divides $r_d (n_1).$
Now \Lem{L:IsosAreEq}
provides an isometric unital \hm{} $\af_1 \colon A_{m_1} \to B_{n_1}.$
We have $\af_{1} \circ \bt_0 = \ps_{n_{1}, n_0}$
because both are $\ld \mapsto \ld \cdot 1$ for $\ld \in \C.$

Use the same reasoning as above to find
$m_2 > m_1$ such that $r_d (n_1)$ divides $r_c (m_2)$ and
an isometric unital \hm{} $\bt_1^{(0)} \colon B_{n_1} \to A_{m_{2}}.$
By \Lem{L:IsosAreEq},
there is an isometry $u_1 \in A_{m_{2}}$
such that
\[
u_1 \big( \big( \bt_1^{(0)} \circ \af_1 \big) (a) \big) u_1^{-1}
  = \ph_{m_{2}, m_1} (a)
\]
for all $a \in A_{m_1}.$
Define $\bt_1 \colon B_{n_1} \to A_{m_{2}}$
by $\bt_1 (b) = u_1 \bt_1^{(0)} (b) u_1^{-1}$ for all $b \in B_{n_1}.$
Then $\bt_1 \circ \af_1 = \ph_{m_{2}, m_1}.$

The choice of $n_2$ and construction of $\af_2$
are similar to the choice of $m_1$ and construction of $\bt_1.$
We will choose any isometric unital \hm{}
$\af_2^{(0)} \colon A_{m_2} \to B_{n_2}$
and take $\af_2 (a) = v_2 \af_2^{(0)} (a) v_2^{-1}$
for a suitable isometry $v_2 \in B_{n_2}.$

Proceed by induction.

The $\af_k$ now combine to yield an isometric \hm{}
from the algebraic direct limit of the~$A_m$
to the algebraic direct limit of the~$B_n.$
This \hm{} extends by continuity to an isometric \hm{}
$\af \colon A \to B.$
Similarly,
we get an isometric \hm{} $\bt \colon B \to A.$
Moreover, one checks that $(\bt \circ \af) (a) = a$
for every $a$ in the algebraic direct limit of the~$A_m,$
whence $\bt \circ \af = \id_A.$
Similarly $\af \circ \bt = \id_B.$
In the same way,
we use the maps
$\af_k \otimes_p \id_D$ and $\bt_k \otimes_p \id_D$
to construct an isometric \hm{}
$\gm \colon \dirlim A_n \otimes_p D \to \dirlim B_n \otimes_p D$
and an isometric \hm{}
$\dt \colon \dirlim B_n \otimes_p D \to \dirlim A_n \otimes_p D$
which is an inverse for~$\gm.$
It is obvious that $\gm \circ \sm = \ta \circ \af$
and that $\dt \circ \ta = \sm \circ \bt.$
It is also obvious that $\gm \circ \ld = \et.$
\end{proof}

The following example describes what we will call
an $L^p$~UHF algebra of (infinite) tensor product type.
We include additional notation which will be useful later.

\begin{exa}\label{E_2Y14_LpUHF}
Let $p \in [1, \I).$
Set ${\mathbb{N}} = \N.$
For each $n \in {\mathbb{N}},$
let $(X_n, {\mathcal{B}}_n, \mu_n)$ be a probability space,
let $d (n) \in \{ 2, 3, \ldots \},$
and let $\rh_n \colon M_{d (n)} \to L (L^p (X_n, \mu_n))$
be a representation
(unital, by our conventions).

For every subset $S \subset {\mathbb{N}},$
define $X_S = \prod_{n \in S} X_n,$
let ${\mathcal{B}}_S$ be the product $\sm$-algebra on~$X_S,$
and let $\mu_S$ be the product measure on~$X_S.$
Thus $X_{\varnothing}$ is a one point space with counting measure.
We let $1_S$ denote the identity operator on $L^p (X_S, \mu_S).$
For $S \subset {\mathbb{N}}$ and $n \in \Nz$ we take
\[
S_{\leq n} = S \cap \{ 1, 2, \ldots, n \}
\andeqn
S_{> n} = S \cap \{ n + 1, \, n + 2, \ldots \}.
\]
In particular,
\[
{\mathbb{N}}_{\leq n} = \{ 1, 2, \ldots, n \}
\andeqn
{\mathbb{N}}_{> n} = \{ n + 1, \, n + 2, \ldots \}.
\]
Then Theorem~\ref{T-LpTP}
provides an obvious identification
\begin{equation}\label{Eq_2Y14_2Fact}
L^p (X_S, \mu_S)
  = L^p (X_{S_{\leq n}}, \, \mu_{S_{\leq n}})
   \otimes_p L^p (X_{S_{> n}}, \, \mu_{S_{> n}}).
\end{equation}

Suppose now that $S$ is finite.
Set $l (S) = \card (S),$
and
write
\[
S = \{ m_{S, 1}, \, m_{S, 2}, \, \ldots, m_{S, l (S)} \}
\]
with $m_{S, 1} < m_{S, 2} < \cdots < m_{S, l (S)}.$
Theorem~\ref{T-LpTP}
provides an obvious identification
\[
L^p (X_S, \mu_S)
 = L^p (X_{m_{S, 1}}, \mu_{m_{S, 1}})
   \otimes_p L^p (X_{m_{S, 2}}, \mu_{m_{S, 2}})
   \otimes_p \cdots
   \otimes_p L^p (X_{m_{S, l (S)}}, \mu_{m_{S, l (S)}}).
\]
Set
\[
d (S) = \prod_{j = 1}^{l (S)} d (m_{S, j})
% d (S) = d (m_{S, 1}) \cdot d (m_{S, 2}) \cdots d (m_{S, l (S)})
\andeqn
% M_S = M_{d (m_{S, 1})} \otimes M_{d (m_{S, 2})}
%     \otimes \cdots \otimes M_{d (m_{S, l (S)})}.
M_S = \bigotimes_{j = 1}^{l (S)} M_{d (m_{S, j})}.
\]
We take $d (\E) = 1$ and $M_{\E} = \C.$
Then $M_S \cong M_{d (S)}.$
Further let $\rh_S \colon M_S \to L (L^p (X_S, \mu_S))$
be the unique representation such that
for
\[
a_1 \in M_{d (m_{S, 1})}, \, \, a_2 \in M_{d (m_{S, 2})},
  \, \, \cdots, \, \, a_{l (S)} \in M_{d (m_{S, l (S)})},
\]
we have,
following Theorem~\ref{T-LpTP}(\ref{T-LpTP-3a}),
\[
\rh_S (a_1 \otimes a_2 \otimes \cdots \otimes a_{l (S)})
  = \rh_{d (m_{S, 1})} (a_1) \otimes \rh_{d (m_{S, 2})} (a_2)
     \otimes \cdots \otimes \rh_{d (m_{S, l (S)})} (a_{l (S)}).
\]

For finite sets $S \subset T \subset {\mathbb{N}},$
there is an obvious \hm{} $\ph_{T, S} \colon M_S \to M_T,$
obtained by filling in a tensor factor of~$1$ for every element
of $T \setminus S.$
Formally,
for
\[
a_1 \in M_{d (m_{S, 1})}, \, \, a_2 \in M_{d (m_{S, 2})},
  \, \, \cdots, \, \, a_{l (S)} \in M_{d (m_{S, l (S)})},
\]
we have
\[
\ph_{T, S} (a_1 \otimes a_2 \otimes \cdots \otimes a_{l (S)})
  = b_1 \otimes b_2 \otimes \cdots \otimes b_{l (T)}
\]
with $b_k = a_j$ when $m_{T, k} = m_{S, j}$
and $b_k = 1$ whenever $m_{T, k} \not\in S.$
We then define
$\rh_{T, S} = \rh_T \circ \ph_{T, S} \colon
   M_S \to L (L^p (X_T, \mu_T)).$
When $S$ is finite but $T$ is not,
choose some $n \geq \sup (S).$
% For any $R \subset {\mathbb{N}},$
% let $1_R$ be the identity operator on $L^p (X_R, \mu_R).$
Following~(\ref{Eq_2Y14_2Fact})
and Theorem~\ref{T-LpTP}(\ref{T-LpTP-3a}),
for $a \in M_S$ set
\[
\rh_{T, S} (a) = \rh_{T_{\leq n}, S} (a) \otimes 1_{T_{> n}}
  \in L (L^p (X_T, \mu_T)).
\]
This clearly defines a \rpn{} of $M_S$
on $L^p (X_T, \mu_T),$
which is independent of the choice of~$n \geq \sup (S).$

We equip $M_S$ with the norm
$\| a \| = \| \rh_S (a) \|$
for all $a \in M_S.$
Then the maps $\ph_{T, S},$
for $S \subset T \subset {\mathbb{N}}$ finite,
and $\rh_{T, S},$
for $S \subset T \subset {\mathbb{N}}$ with $S$ finite,
are all isometric,
by Proposition~\ref{P:TProdOfAlgs}(\ref{P:TProdOfAlgs-2})
(using Lemma~\ref{L_2Z29_PermFact} to reorder tensor factors as needed).
For $S \subset {\mathbb{N}}$ finite,
we now define $A_S \subset \LLp$
by $A_S = \rh_{{\mathbb{N}}, S} (M_S),$
and for $S \subset {\mathbb{N}}$ infinite,
we set
\[
A_S = {\overline{\bigcup_{n = 0}^{\I} A_{S_{\leq n}} }}.
\]
In a similar way,
for arbitrary $S, T \subset {\mathbb{N}}$ with $S \subset T,$
we define $A_{T, S} \subset L (L^p (X_T, \mu_T)).$

The algebra $A = A_{{\mathbb{N}}}$ is an $L^p$~UHF algebra,
of type~$N_d$
in the notation of Definition~\ref{D_2Y15_SNat}
and the sense of Definition~\ref{D_2Y14_pUHFTypeN}.
When the ingredients used in its construction need to be specified,
we set $\rh = (\rh_1, \rh_2, \ldots)$
and write $A (d, \rh).$
\end{exa}

\begin{dfn}\label{D_2Y15_ITP}
Let $\XBM$ be a \sfm,
let $p \in [1, \infty),$
and let $A \subset \LLp$ be a unital subalgebra.
\begin{enumerate}
\item\label{D_2Y15_ITP_Basic}
We say that $A$ is an
{\emph{$L^p$~UHF algebra of tensor product type}}
if there exist $d$ and $\rh = (\rh_1, \rh_2, \ldots)$
as in Example~\ref{E_2Y14_LpUHF}
such that $A$ is isometrically isomorphic to $A (d, \rh).$
\item\label{D_2Y15_ITP_Spatial}
We say that $A$ is a {\emph{spatial $L^p$~UHF algebra}}
if in addition it is possible to choose
each \rpn{} $\rh_n$ to be spatial in the sense
of Definition~\ref{D_2Y14_SpMd}.
\item\label{D_2Y15_ITP_LEI}
We say that $A$ {\emph{locally has enough isometries}}
if it is possible to choose $d$ and $\rh$
as in~(\ref{D_2Y15_ITP_Basic})
such that, in addition,
$\rh_n$ locally has enough isometries,
in the sense of Definition~\ref{D_2Y15_EnId}(\ref{D_2Y15_EnId_Loc}),
for all $n \in \N.$
\item\label{D_2Y15_DomSpatial}
If $A$ locally has enough isometries,
we further say that $A$
{\emph{dominates the spatial representation}}
if it is possible to choose $d$ and $\rh$
as in~(\ref{D_2Y15_ITP_LEI})
such that, in addition,
for all $n \in \N$ the representation
$\rh_n$ dominates the spatial representation
in the sense of
Definition~\ref{D_2Y15_EnId}(\ref{D_2Y15_EnId_Compat}).
\end{enumerate}
\end{dfn}

\begin{thm}\label{T_2Y15_SpEx}
Let $p \in [1, \infty)$
and let $N$ be a supernatural number.
Then:
\begin{enumerate}
\item\label{T_2Y15_SpEx_Ex}
There exists a spatial direct system of type~$N,$
as in Definition~\ref{D:StdUHF}.
\item\label{T_2Y15_SpEx_LimIsSp}
The direct limit of any spatial direct system of type~$N$
is a spatial $L^p$~UHF algebra of type~$N.$
\item\label{T_2Y15_SpEx_SpIsLim}
Every spatial $L^p$~UHF algebra of type~$N$
is the direct limit of a spatial direct system of type~$N.$
\item\label{T_2Y15_SpEx_SpEI}
Every spatial $L^p$~UHF algebra of type~$N$
locally has enough isometries.
\item\label{T_2Y15_SpEx_Contr}
Let $A$ be a spatial $L^p$~UHF algebra of type~$N,$
let $\XBM$ be a \sfm,
and let $\rh \colon A \to \LLp$ be a unital contractive \hm.
Then $\rh$ is isometric.
\end{enumerate}
\end{thm}

\begin{proof}
We first prove~(\ref{T_2Y15_SpEx_SpIsLim}).
Use the notation of Example~\ref{E_2Y14_LpUHF},
so that we are assuming that $\rh_n$ is spatial
for all $n \in {\mathbb{N}},$
and that $N = N_d.$
It is enough to prove that $\ph_{T, S}$ and $\rh_{T, S}$ are isometric
whenever $S \subset T \subset {\mathbb{N}}$
with $S$ finite,
and that $\rh_S,$
regarded as a map $\MP{d (S)}{p} \to L ( L^p (X_S, \mu_S)),$
is isometric whenever $S \subset {\mathbb{N}}$ is finite.
We already saw in Example~\ref{E_2Y14_LpUHF}
that $\ph_{T, S}$ and $\rh_{T, S}$ are isometric.
That $\rh_S$ is isometric follows from Corollary~\ref{C_2Y21MpTens}.
We have proved~(\ref{T_2Y15_SpEx_SpIsLim}).

Now
choose a sequence $d$ as in Definition~\ref{D_2Y15_SNat}
such that $N_d = N.$
In Example~\ref{E_2Y14_LpUHF},
take $X_n = \{ 1, 2, \ldots, d (n) \}.$
Let $\mu_n$ be normalized counting measure on~$X_n,$
that is, $\mu_n$ is the probability measure which gives the same mass
to every point.
Multiplication by $d (n)^{- 1/p}$ defines an
isometric isomorphism from $L^p (X_n, \mu_n)$ to the space
$l_{d (n)}^p$ of Notation~\ref{N:FDP},
and hence gives an isometric isomorphism
from $\MP{d (n)}{p}$ to~$L (L^p (X_n, \mu_n)).$
For $n \in {\mathbb{N}},$
we take $\rh_n$ to be this isomorphism.
Then $A (d, \rh)$ is a spatial $L^p$~UHF algebra of type~$N.$
Applying part~(\ref{T_2Y15_SpEx_SpIsLim}),
we find that $A (d, \rh)$
is the direct limit of a spatial direct system of type~$N.$
This is~(\ref{T_2Y15_SpEx_Ex}).

By Theorem~\ref{T:UHFIsoCrit}
(with $D = \C$),
the direct limit of any spatial direct system of type~$N$
is isometrically isomorphic to $A (d, \rh).$
So part~(\ref{T_2Y15_SpEx_LimIsSp}) follows.
Similarly, for~(\ref{T_2Y15_SpEx_SpEI}),
we need only consider $A (d, \rh),$
which locally has enough isometries by Example~\ref{E_2Y15_MdIdRep}.

For part~(\ref{T_2Y15_SpEx_Contr}),
it now suffices to show that
whenever $\YCN$ is a \sfm{}
and $\sm \colon A (d, \rh) \to L (L^p (Y, \nu))$
is a unital contractive \hm,
then $\sm$ is isometric.
For every $n \in {\mathbb{N}},$
the subalgebra $A_n$ is isometrically isomorphic
to $\MP{ d ({\mathbb{N}}_{\leq n}) }{p}.$
If $p \neq 2,$ then $\sm |_{A_n}$ is isometric by
Theorem~\ref{T_2Y17Iso}.
If $p = 2,$ use Lemma~\ref{L_3717_CSt}
and simplicity of $\MP{ d ({\mathbb{N}}_{\leq n}) }{p}$
to see that $\sm |_{A_n}$ is isometric.
Since $\bigcup_{n = 1}^{\I} A_n$ is dense in $A (d, \rh),$
it follows that $\sm$ is isometric.
\end{proof}

We now prove that an $L^p$~UHF algebra
of tensor product type which locally has enough isometries is simple.
For a spatial $L^2$~UHF algebra,
when the algebras are all \ca s,
this is automatic:
the direct limit of simple \ca s is simple.
The proof depends on the fact that an injective \hm{} of \ca s
is isometric,
which is not true for Banach algebras.
We suppose that the direct limit of simple Banach algebras
need not be simple,
although we do not know any counterexamples.
(See Question~\ref{Q_3121_NSimp}.)
Spatial $L^p$~UHF algebras are in fact simple,
but the proof requires more work.

\begin{lem}\label{L:UHFCondExpt}
Let $p \in [1, \I),$
and let $A$ be an $L^p$~UHF algebra of tensor product type
which locally has enough isometries,
as in Definition~\ref{D_2Y15_ITP}.
Let the notation be as in Example~\ref{E_2Y14_LpUHF},
with the choices made so that $\rh_n$
locally has enough isometries for all $n \in {\mathbb{N}}.$
Then for every $R, S \subset {\mathbb{N}}$ with $R \subset S,$
there is a unique continuous linear map
$E_{R, S} \colon A_S \to A_R$
such that for every finite set $Q \subset S \setminus R,$
every $b \in M_Q,$ and every $a \in A_R,$ we have
\begin{equation}\label{Eq:PartialCE}
E_{R, S} (\rh_{{\mathbb{N}}, Q} (b) a) = \tr (b) a.
\end{equation}
Moreover:
\begin{enumerate}
\item\label{L:UHFCondExpt_IdP}
$E_{R, S}$ is an ideal preserving
Banach conditional expectation
for all $R, S \subset {\mathbb{N}}$ with $R \subset S.$
\item\label{L:UHFCondExpt_Cmp}
Whenever $R \subset T \subset S \subset {\mathbb{N}},$
we have $E_{R, T} \circ E_{T, S} = E_{R, S}.$
\item\label{L:UHFCondExpt_Restr}
Whenever $R \subset S \subset {\mathbb{N}}$ and $n \in \Nz,$
we have
$E_{R, S} |_{A_{S_{\leq n}}} = E_{R_{\leq n}, S_{\leq n} }.$
\end{enumerate}
\end{lem}

\begin{proof}
Uniqueness follows from the fact that the elements
$\rh_{{\mathbb{N}}, Q} (b) a$
which appear in~(\ref{Eq:PartialCE})
span a dense subset of~$A_S.$

Next,
suppose $R \subset T \subset S \subset {\mathbb{N}}$
and $E_{R, T},$ $E_{T, S},$ and $E_{R, S}$
are all known to exist.
We prove~(\ref{L:UHFCondExpt_Cmp}).
Let $P \subset S \setminus T$ and $Q \subset T \setminus R$ be finite,
and let $a \in A_R,$ $b \in M_P,$ and $c \in M_Q.$
Using~(\ref{Eq:PartialCE}) repeatedly, we have
\begin{align*}
(E_{R, T} \circ E_{T, S}) \big( \rh_{{\mathbb{N}}, P \cup Q} (b \otimes c) a \big)
& = (E_{R, T} \circ E_{T, S})
        \big( \rh_{{\mathbb{N}}, P} (b) \rh_{{\mathbb{N}}, Q} (c) a \big)
   \\
& = E_{R, T} \big( \tr (b) \rh_{{\mathbb{N}}, Q} (c) a \big)
  = \tr (b) \tr (c) a
   \\
& = \tr (b \otimes c) a
  = E_{R, S} \big( \rh_{{\mathbb{N}}, P \cup Q} (b \otimes c) a \big).
\end{align*}
Since the elements $\rh_{{\mathbb{N}}, P \cup Q} (b \otimes c) a$
span a dense subset of~$A_S,$
part~(\ref{L:UHFCondExpt_Cmp})~follows.

We now prove existence.
When $S \setminus R$ is finite,
all but~(\ref{L:UHFCondExpt_Cmp})
is immediate from Lemma~\ref{L:CondExptMd}
and the definitions,
and (\ref{L:UHFCondExpt_Cmp})~follows as above.
In the general case,
for $n \in \Nz$ define
$E_n = E_{R_{\leq n}, S_{\leq n}} \colon
    A_{S_{\leq n}} \to A_{R_{\leq n}}.$
For $m, n \in \Nz$ with $m \leq n,$
applying~(\ref{L:UHFCondExpt_Restr}),
with $R_{\leq n},$ $S_{\leq n},$ and~$m$
in place of $R,$ $S,$ and~$n,$
we get $E_n |_{A_{S_{\leq m}}} = E_m.$
Therefore there is a linear map
\[
E_{\I} \colon
 \bigcup_{n = 0}^{\I} A_{S_{\leq n}}
    \to \bigcup_{n = 0}^{\I} A_{R_{\leq n}}
\]
satisfying~(\ref{Eq:PartialCE})
for every finite set $Q \subset S \setminus R,$
every $b \in M_Q,$
and every $a \in \bigcup_{n = 0}^{\I} A_{S_{\leq n}}.$
Since $\| E_n \| \leq 1$ for all $n \in \Nz,$
the map $E_{\I}$ extends to a linear map
$E_{R, S} \colon A_S \to A_R$
by continuity,
such that $\| E_{R, S} \| \leq 1$
and $E_{R, S}$ satisfies~(\ref{Eq:PartialCE})
for for every finite set $Q \subset S \setminus R,$
every $b \in M_Q,$
and every $a \in A_S.$
The properties of a Banach conditional expectation
follow by continuity.
The relation (\ref{L:UHFCondExpt_Cmp})~follows as above
and~(\ref{L:UHFCondExpt_Restr}) is clear,
but it remains to prove that $E_{R, S}$ is ideal preserving.

We first claim that if $a \in A_S,$
then $\limi{n} E_{R \cup S_{> n}, \, S} (a) = E_{R, S} (a).$
To see this,
choose $N \in {\mathbb{N}}$ and $b \in A_{S_{\leq N}}$
such that $\| b - a \| < \tfrac{\ep}{2}.$
For $n \geq N,$
we have, using~(\ref{L:UHFCondExpt_Restr}) twice,
$E_{R \cup S_{> n}, \, S} (b)
 = E_{R_{\leq N}, S_{\leq N} } (b) = E_{R, S} (b).$
Therefore,
using $\| E_{R \cup S_{> n}, \, S} \| \leq 1$
and $\| E_{R, S} \| \leq 1,$ we have
\[
\| E_{R \cup S_{> n}, \, S} (a) - E_{R, S} (a) \|
  \leq 2 \| a - b \|
  < \ep.
\]
The claim follows.

Now let $I \subset A_S$ be a closed ideal.
For all $n \in {\mathbb{N}},$
we have $S \setminus (R \cup S_{> n}) \subset {\mathbb{N}}_{\leq n},$
which is finite.
So $a \in I$ implies $E_{R \cup S_{> n}, \, S} (a) \in I.$
Since $I$ is closed,
the claim implies that $E_{R, S} (a) \in I,$
as desired.
This completes the proof.
\end{proof}

\begin{cor}\label{C:LpUHFTrace}
Let $p \in [1, \I),$
and let $A$ be an $L^p$~UHF algebra of tensor product type
which locally has enough isometries,
as in Definition~\ref{D_2Y15_ITP}.
Then there exists a unique \ct{} linear functional
$\ta \colon A \to \C$ such that $\ta (1) = 1$
and $\ta (b a) = \ta (a b)$ for all $a, b \in A.$
Moreover, $\| \ta \| = 1.$
\end{cor}

We don't know whether uniqueness still holds without the
continuity assumption.
(Without the adjoint and positivity of~$\ta,$
we don't know any automatic continuity results.)

\begin{proof}[Proof of Corollary~\ref{C:LpUHFTrace}]
Use the notation of Example~\ref{E_2Y14_LpUHF}.
For existence and $\| \ta \| = 1,$
apply \Lem{L:UHFCondExpt} with $R = \E$ and $S = {\mathbb{N}}.$
Take $\ta$ to be defined by $E_{\E, {\mathbb{N}}} (a) = \ta (a) \cdot 1$
for $a \in A.$
For uniqueness,
use the notation of \Lem{L:UHFCondExpt}.
By uniqueness of the trace on a full matrix algebra,
for all $n \in \Nz$
and $a \in M_{ {\mathbb{N}}_{\leq n} } = M_{r_d (n)}$
we must have
$\ta ( \rh_{{\mathbb{N}}, {\mathbb{N}}_{\leq n} } (a)) = \tr (a).$
The uniqueness statement in \Lem{L:UHFCondExpt}
now implies that $\ta (a) \cdot 1 = E_{\E, {\mathbb{N}}} (a)$
for all $a \in A.$
\end{proof}

\begin{thm}\label{T:LpUHFSimple}
Let $p \in [1, \I),$
and let $A$ be an $L^p$~UHF algebra of tensor product type
which locally has enough isometries,
as in Definition~\ref{D_2Y15_ITP}.
Then $A$ is algebraically simple.
\end{thm}

\begin{proof}
Let the notation be as in Example~\ref{E_2Y14_LpUHF}
and Lemma~\ref{L:UHFCondExpt}.
For $n \in \Nz,$
set $A_n = A_{ {\mathbb{N}}_{\leq n} },$
and let
$E_n = E_{ {\mathbb{N}}_{\leq n}, {\mathbb{N}}} \colon A \to A_n.$
Let $I \subset A$ be a closed ideal.
Since $E_n$ is ideal preserving (by Lemma~\ref{L:UHFCondExpt}),
it follows from Lemma~\ref{L_2Y15_IdPLim}
that $I = {\overline{\bigcup_{n = 0}^{\I} (I \cap A_n)}}.$
If $I \cap A_n = 0$ for all $n \in \Nz,$
we thus get $I = 0.$
Otherwise, there is $n \in \Nz$ such that $I \cap A_n \neq 0.$
Since $A_n \cong M_{r_d (n)}$ is simple,
it follows that $1 \in I \cap A_n \subset I.$
So $I = A.$
\end{proof}

\begin{qst}\label{Q_3121_NSimp}
Let $p \in [1, \I).$
Is there an $L^p$~UHF algebra which is not simple?
Is there an $L^p$~UHF algebra of tensor product type
which is not simple?
\end{qst}

It seems plausible that there is at least
some nonsimple $L^p$~UHF algebra,
but the opposite outcome would also not be surprising.
New methods seem to be required to approach this question.

% This follows directly from
% Example 2.3.16,
% Proposition 2.1.17,
% and Theorem 2.2.4 of~\cite{Rnd}.

Finally,
we consider amenability.
The definition of an amenable Banach algebra is given in
Definition 2.1.9 of~\cite{Rnd};
see Theorem 2.2.4 of~\cite{Rnd}
for two standard equivalent conditions.

\begin{thm}\label{T_3714Amen}
Let $p \in [1, \I).$
Then any spatial $L^p$~UHF algebra~$A$ is an amenable Banach algebra.
\end{thm}

\begin{proof}
By Example 2.3.16 of~\cite{Rnd},
the algebra $\MP{d}{p}$ is $1$-amenable
(in the sense of Definition 2.3.15 of~\cite{Rnd})
for all $d \in \N.$
Theorem~\ref{T_2Y15_SpEx}(\ref{T_2Y15_SpEx_SpIsLim})
implies that $A$ is a direct limit of a direct system
of such algebras, with isometric maps.
Proposition 2.3.17 of~\cite{Rnd}
therefore implies that $A$ is $1$-amenable, hence amenable.
\end{proof}

In~\cite{Ph-Lp2b},
for each $p \in [1, \I)$ and each supernatural number~$N,$
we exhibit a class ${\mathcal{C}}_{p, N}$
of $L^p$~UHF algebras of tensor product type
with supernatural number~$N,$
which contains the spatial $L^p$~UHF algebra~$A$ of type~$N,$
which contains uncountably many algebras
which are pairwise nonisomorphic as Banach algebras,
and such that an algebra $B \in {\mathcal{C}}_{p, N}$
is amenable \ifo{} $B \cong A.$

\section{Classification of spatial $L^p$~UHF algebras}\label{Sec_Diffp}

\indent
In this section,
we give the isomorphism classification of spatial $L^p$~UHF algebras.
Briefly,
let $p_1, p_2 \in [1, \I),$
let $A_1$ be a spatial $L^{p_1}$~UHF algebra,
and let $A_2$ be a spatial $L^{p_2}$~UHF algebra.
Then $A_1 \cong A_2$
\ifo{} $K_0 (A_1) \cong K_0 (A_2)$
as scaled ordered groups
(equivalently, $A_1$ and $A_2$ have the same supernatural number)
and $p_1 = p_2.$
This is a generalization of the usual classification of UHF \ca{s}.

None of the material here is needed in Section~\ref{Sec:Simplicity}.

The following theorem will be proved in~\cite{Ph7}.
It is not difficult at this stage;
it is just the fact that K-theory commutes with direct limits
of Banach algebras.
We postpone the proof to avoid taking up space here
with a discussion of K-theory.

\begin{thm}\label{T-KThyUHF}
Let $p \in [1, \I).$
Let $N$ be a supernatural number
(Definition~\ref{D_2Y15_SNat})
and let $D$ be an $L^p$~UHF algebra of type~$N$
(Definition~\ref{D_2Y14_pUHFTypeN}).
For $n \in \Nz$ define
\[
k (n) = p_1^{\min (N (1), \, n)} p_2^{\min (N (2), \, n)}
          \cdots p_n^{\min (N (n), \, n)}.
\]
Then $K_1 (D) = 0$ and
$K_0 (D)
 \cong \bigcup_{n = 1}^{\I}
     k (n)^{-1} \Z
 \subset \Q,$
via an order isomorphism which sends the class $[1_D]$
of the idempotent $1_D \in D$ to $1 \in \Q.$
\end{thm}

As a corollary, we get the K-theoretic classification of
spatial $L^p$~UHF algebras.

\begin{thm}\label{T-ClassUHF}
Let $p \in [1, \I).$
Let $N_1$ and $N_2$ be supernatural numbers
(Definition~\ref{D_2Y15_SNat}).
Let $D_1$ be a spatial $L^p$~UHF algebra of type~$N_1$
and let $D_2$ be a spatial $L^p$~UHF algebra of type~$N_2.$
Then \tfae:
\begin{enumerate}
\item\label{T-ClassUHF-AlgIso}
There is an isometric algebra isomorphism from $D_1$ to~$D_2.$
\item\label{T-ClassUHF-AlgHme}
There is a \ct{} algebra isomorphism from $D_1$ to~$D_2.$
\item\label{T-ClassUHF-RingIso}
$D_1$ and $D_2$ are isomorphic as rings.
\item\label{T-ClassUHF-KIso}
There is an isomorphism $\et \colon K_0 (D_1) \to K_0 (D_2)$
such that $\et ([1_{D_1}]) = [1_{D_2}].$
\item\label{T-ClassUHF-Type}
$N_1 = N_2.$
\end{enumerate}
\end{thm}

\begin{proof}
It is obvious that
(\ref{T-ClassUHF-AlgIso}) implies~(\ref{T-ClassUHF-AlgHme})
and that
(\ref{T-ClassUHF-AlgHme}) implies~(\ref{T-ClassUHF-RingIso}).
That
(\ref{T-ClassUHF-RingIso}) implies~(\ref{T-ClassUHF-KIso})
follows from the fact that $K_0$ of a Banach algebra
only depends on the ring structure.
Using Theorem~\ref{T-KThyUHF},
it is standard (and easy to check)
that (\ref{T-ClassUHF-KIso}) implies~(\ref{T-ClassUHF-Type}).
The implication from (\ref{T-ClassUHF-Type})
to~(\ref{T-ClassUHF-AlgIso})
is Theorem~\ref{T:UHFIsoCrit}
(with $D = \C$ there).
\end{proof}

\begin{thm}\label{T_2Y19UNOI}
Let $p \in [1, \I).$
Let $N_1$ and $N_2$ be supernatural numbers
(Definition~\ref{D_2Y15_SNat}).
Let $D_1$ be an $L^p$~UHF algebra of type~$N_1$
and let $D_2$ be an $L^p$~UHF algebra of type~$N_2.$
If $N_1 \neq N_2,$
then $D_1 \not\cong D_2.$
\end{thm}

We do not assume that $D_1$ or $D_2$ is spatial.

\begin{proof}[Proof of Theorem~\ref{T_2Y19UNOI}]
This is an easy consequence of Theorem~\ref{T-KThyUHF}.
\end{proof}

We now turn to the comparison of $L^p$~UHF algebras
for different values of~$p.$
For distinct $p_1, p_2 \in [1, \I),$
no spatial $L^{p_1}$~UHF algebra
can be isomorphic to any spatial $L^{p_2}$~UHF algebra.
For some choices of $p_1$ and~$p_2,$
no spatial $L^{p_1}$~UHF algebra
can be isomorphic to any $L^{p_2}$~UHF algebra,
spatial or not.

Our results are not as satisfactory as those for $\OP{d}{p}$
in~\cite{Ph5},
where it is shown (Corollary~9.3 of~\cite{Ph5})
that for different values of~$p$
there are no nonzero \ct{} \hm{s}
between such algebras.
There, a \rpn{} on a Banach space~$E$
gave an isomorphic embedding of $l^p (\N)$ in~$E$
(Lemma~9.1 of~\cite{Ph5}).
Here, we get something weaker:
$l^p (\N)$ is crudely finitely representable in~$E$
in the sense of Definition~\ref{D_2Y19FinRep} below.
This property is still strong enough to show that
if $A_j$ is a spatial $L^{p_j}$~UHF algebra for $j = 1, 2,$
with $p_1 \neq p_2,$
then in at least one direction
there is no nonzero \ct{} \hm{}
between $A_1$ and~$A_2.$

\begin{dfn}[Definitions 11.1.1 and 11.1.5
   of~\cite{AK}]\label{D_2Y19FinRep}
Let $E$ and $F$ be infinite dimensional Banach spaces,
and let $M > 1.$
Then $E$ is
{\emph{crudely finitely representable in $F$ with constant~$M$}}
if for every \fd{} subspace $E_0 \subset E$
there exist a \fd{} subspace $F_0 \subset F$
and an invertible linear map $a \colon E_0 \to F_0$
such that $\| a \| \cdot \| a^{-1} \| < M.$
For real Banach spaces we use real linear maps~$a,$
and for complex Banach spaces we use complex linear maps~$a.$
We say that $E$ is
{\emph{finitely representable in $F$}}
if $E$ is crudely finitely representable in $F$ with constant~$M$
for every $M > 1.$
\end{dfn}

We state several results
that are essentially in Chapter~11 of~\cite{AK}.
We are grateful to Bill Johnson for pointing out this reference.
We warn that the standing assumption in~\cite{AK}
is that all Banach spaces are real unless otherwise specified.
In particular, the theorems we cite,
as stated,
apply only to real Banach spaces.

\begin{lem}\label{L_2Y19lpFRep}
Let $p \in [1, \I).$
Let $F$ be a Banach space
and suppose that there are $M_1, M_2 \in (0, \I)$
such that for every $d \in \N$ there is a linear map
$a \colon l_d^p \to F$
such that for all $\xi \in l_d^p$ we have
\[
M_1^{-1} \| \xi \| \leq \| a \xi \| \leq M_2 \| \xi \|.
\]
Then for every $\ep > 0,$
the space $l^p (\N)$
is crudely finitely representable in $F$ with constant $M_1 M_2 + \ep.$
\end{lem}

\begin{proof}
The hypotheses imply that
the condition of Definition~\ref{D_2Y19FinRep}
is satisfied with $E = l^p (\N)$ and
with $M = M_1 M_2 + \frac{\ep}{2},$
except that $E_0$ is restricted to being the span
of finitely many standard basis vectors in $l^p (\N).$
Lemma 11.1.6(ii) of~\cite{AK}
is also valid in the complex case (with the same proof),
and now implies that $l^p (\N)$
is crudely finitely representable in $F$ with constant $M_1 M_2 + \ep.$
\end{proof}

\begin{thm}\label{T_2Y19NFinR}
Let $p, r \in [1, \I)$ with $p \neq r.$
Suppose that there are~$M$
and a \sfm{} $\XBM$
such that $L^r (X, \mu)$ is separable
and $l^p (\N)$ is crudely finitely representable in $L^r (X, \mu)$
with constant~$M.$
Then $1 \leq r < p \leq 2$ or $2 = p < r.$
\end{thm}

\begin{proof}
To distinguish real and complex $L^p$~spaces,
we will write $L_{\R}^p (X, \mu)$ and $l_{\R}^p (S)$
for the real versions of these spaces.
Also, in this proof,
all linear maps will be assumed to be merely real linear.

Since $l_{\R}^p (\N)$ is real isometrically isomorphic to
a subspace of $l^p (\N),$
it follows that, regarding both spaces as real Banach spaces,
$l_{\R}^p (\N)$ is crudely finitely representable in $L^r (X, \mu)$
in the real sense,
with constant~$M.$

{}From now through the end of the proof,
all Banach spaces will be taken as real,
and (crudely) finitely representable will be in the real sense.
Let $\mu \amalg \mu$ be the obvious measure on
the disjoint union $X \amalg X$ of two copies of~$X.$
The space $L^r (X, \mu)$ is isomorphic
(not necessarily isometrically)
to $L_{\R}^r (X \amalg X, \, \mu \amalg \mu).$
Therefore $l_{\R}^p (\N)$ is crudely finitely representable
in $L_{\R}^r (X \amalg X, \, \mu \amalg \mu)$
(with a different constant).

Since $L_{\R}^r (X \amalg X, \, \mu \amalg \mu)$ is separable,
it follows from Proposition~1 in Chapter III.A of~\cite{Wj}
that $L_{\R}^r (X \amalg X, \, \mu \amalg \mu)$ is isometrically
isomorphic to a subspace of the $L^r$~sum
$L_{\R}^r ([0, 1]) \oplus_r l_{\R}^r (\N).$
% (The proof is written for the complex case,
% but only depends on the structure of the measure,
% so applies equally well in the real case.)
Proposition 11.1.7 of~\cite{AK}
tells us that $L_{\R}^r ([0, 1])$
is finitely representable in $l_{\R}^r (\N),$
and it is now immediate to check that
$L_{\R}^r ([0, 1]) \oplus_r l_{\R}^r (\N)$
is finitely representable in
$l_{\R}^r (\N) \oplus_r l_{\R}^r (\N) \cong l_{\R}^r (\N).$
It follows (by essentially the same proof
as that of Proposition 11.1.4 of~\cite{AK})
that $l_{\R}^p (\N)$ is crudely finitely representable
in $l_{\R}^r (\N)$
(with some constant).
By Proposition 11.1.13 of~\cite{AK},
there is a Banach space $E$ isomorphic (not necessarily isometrically)
to $l_{\R}^p (\N)$
which is finitely representable in $l_{\R}^r (\N),$
and Theorem 11.1.8 of~\cite{AK}
then implies that $E$ is isometrically isomorphic to
a subspace of $L_{\R}^r ([0, 1]).$
That is, $l_{\R}^p (\N)$ can be isomorphically embedded
in $L_{\R}^r ([0, 1]).$
The conditions on $p$ and $r$ in the statement
now follow from Theorem 6.4.19 of~\cite{AK}.
\end{proof}

\begin{thm}\label{T_2Y18LpInUHF}
Let $p \in [1, \I).$
Let $A$ be a spatial $L^p$~UHF algebra,
let $E$ be a Banach space,
and let $\ps \colon A \to L (E)$
be a nonzero \ct{} (not necessarily unital) \hm.
Then for every $\ep > 0,$
the space $l^p (\N)$
is crudely finitely representable in $E$
with constant $\| \ps \|^2 + \ep.$
\end{thm}

\begin{proof}
We verify the condition of Lemma~\ref{L_2Y19lpFRep}
with $M_1 = M_2 = \| \ps \|.$

We first claim that if $d \in \N$ and $\ps \colon \MP{d}{p} \to E$
is a nonzero \hm,
then there exists $v \in L (l_d^p, E)$
such that for all $\xi \in l_d^p$ we have
\[
\| \ps \|^{-1} \cdot \| \xi \|
 \leq \| v \xi \|
 \leq \| \ps \| \cdot \| \xi \|.
\]
We follow the method of proof of Lemma~9.1 of~\cite{Ph5}.
Let $q \in (1, \I]$ satisfy
$\frac{1}{p} + \frac{1}{q} = 1.$
Let $(e_{j, k})_{j, k = 1, 2, \ldots, d}$
be the standard system of matrix units for~$M_d.$
For $\ld = (\ld_1, \ld_2, \ldots, \ld_d) \in \C^d,$
we define $s (\ld), t (\ld) \in \MP{d}{p}$ by
\[
s (\ld) = \sum_{j = 1}^d \ld_j e_{j, 1}
\andeqn
t (\ld) = \sum_{j = 1}^d \ld_j e_{1, j}.
\]
For $\et, \mu \in \C^d,$
let $\om_{\et} \colon \C^d \to \C$
be the linear functional
$\om_{\et} (\xi) = \sum_{j = 1}^d \et_j \xi_j$
for $\xi \in \C^d,$
and let $b_{\mu, \et} \in M_d$ be the rank one operator given by
$b_{\mu, \et} \xi = \om_{\et} (\xi) \mu$
for $\xi \in \C^d.$
(See Example~2.5 of~\cite{Ph5}.)
% (See Example~\ref{E:NormRankOne}.)
Taking $\et = (1, 0, \ldots, 0)$ and $\mu = \ld$
gives $b_{\mu, \et} = s (\ld),$
and taking $\et = \ld$ and $\mu = (1, 0, \ldots, 0)$
gives $b_{\mu, \et} = t (\ld).$
According to Example~2.5 of~\cite{Ph5},
% According to Example~\ref{E:NormRankOne},
we therefore have
\[
\| s (\ld) \|
 = \| \ld \|_p \| (1, 0, \ldots, 0) \|_q
 = \| \ld \|_p
\andeqn
\| t (\ld) \|
 = \| (1, 0, \ldots, 0) \|_p \| \ld \|_q
 = \| \ld \|_q.
\]
(These can also be checked directly.)
Also, if $\ld, \mu \in \C^d,$
then
$t (\mu) s (\ld) = \om_{\mu} ( \ld ) e_{1, 1}.$
% $t (\mu) s (\ld)
%   = \left( \sum_{j = 1}^d \mu_j \ld_j \right) e_{1, 1}.$

The element $\ps (e_{1, 1}) \in L (E)$
is an idempotent,
and is nonzero since $\ps \neq 0$ and $\MP{d}{p}$ is simple.
Fix $\et_0 \in E$
with $\| \et_0 \| = 1$ and $\ps (e_{1, 1}) \et_0 = \et_0.$
Define $v \colon \C^d \to E$
by
$v (\ld) = \ps ( s (\ld)) \et_0$
for $\ld \in \C^d.$

Let $\ld \in \C^d.$
Then
\[
\| v (\ld) \|
 \leq \| \ps \| \cdot \| s (\ld) \| \cdot \| \et_0 \|
 = \| \ps \| \cdot \| \ld \|_p.
\]
Also, it is well known
that there exists $\gm \in \C^d \setminus \{ 0 \}$
such that
% %
% \begin{equation}\label{Eq:EvTo1}
% \sum_{j = 1}^{d} \gm_j \ld_j = 1
% \end{equation}
% %
% and
% %
% \begin{equation}\label{Eq:NormGm}
% \| \gm \|_{q} = \| \ld \|_{p}^{-1}.
% \end{equation}
% %
%
% \[
% \sum_{j = 1}^{d} \gm_j \ld_j = 1
% \andeqn
% \| \gm \|_{q} = \| \ld \|_{p}^{-1}.
% \]
$\om_{\gm} (\ld) = 1$ and $\| \gm \|_{q} = \| \ld \|_{p}^{-1}.$
(The method of proof can be found,
for example, at the beginning of Section~6.5 of~\cite{Ry}.)
Therefore
\begin{align*}
1
& = \| \et_0 \|
  = \big\| \ps \big( t (\gm) s (\ld) \big) \et_0 \big\|
  = \| \ps ( t (\gm) ) v (\ld) \|
     \\
& \leq \| \ps \| \cdot \| t (\gm) \| \cdot \| v (\ld) \|
  = \| \ps \| \cdot \| \gm \|_{q} \cdot \| v (\ld) \|
  = \| \ps \| \cdot \| \ld \|_{p}^{-1} \cdot \| v (\ld) \|.
\end{align*}
So $\| \ps \|^{-1} \cdot \| \ld \|_{p} \leq \| v (\ld) \|.$
This completes the proof of the claim.

We now return to our spatial $L^p$~UHF algebra~$A.$
It follows from Theorem~\ref{T_2Y15_SpEx}(\ref{T_2Y15_SpEx_SpIsLim})
and Definition~\ref{D:StdUHF}
that $A$ contains isometric copies of $\MP{d}{p}$
for arbitrarily large $d \in \Z,$
and hence for all $d \in \N.$
Applying the claim to the restriction of $\ps$ to these subalgebras,
for all $d \in \N$ we get $v_d \in L (l_d^p, E)$
such that for all $\xi \in l_d^p$ we have
\[
\| \ps \|^{-1} \cdot \| \xi \|
 \leq \| v_d \xi \|
 \leq \| \ps \| \cdot \| \xi \|.
\]
The statement of the theorem now follows from Lemma~\ref{L_2Y19lpFRep}.\end{proof}

We suppose that there is no converse to Theorem~\ref{T_2Y18LpInUHF},
but we do not know a counterexample.

\begin{thm}\label{T_UHFNonIsoP}
Let $p_1, p_2 \in [1, \I)$ be distinct.
Let $A_1$ be a spatial $L^{p_1}$~UHF algebra
and let $A_2$ be a spatial $L^{p_2}$~UHF algebra.
Then $A_1 \not\cong A_2.$
\end{thm}

We emphasize:
There isn't even a not necessarily isometric isomorphism.

\begin{proof}[Proof of Theorem~\ref{T_UHFNonIsoP}]
By exchanging $A_1$ and $A_2$ if necessary,
we can assume that $p_1 \neq 2$
and that, if $p_2 \neq 2,$
then $p_1 < p_2.$
Now Theorem~\ref{T_2Y19NFinR}
implies that
there is no \sfm{} $\XBM$
such that $L^{p_2} (X, \mu)$ is separable
and $l^{p_1} (\N)$ is
crudely finitely representable in $L^{p_2} (X, \mu)$
for any constant~$M.$
It follows from Definition~\ref{D_2Y15_ITP}
and the construction in Example~\ref{E_2Y14_LpUHF}
that there is a \sfm{} $\XBM$
such that $L^{p_2} (X, \mu)$ is separable
and $A_2$ is isometrically isomorphic to a closed subalgebra
of $L (L^{p_2} (X, \mu)).$
Therefore Theorem~\ref{T_2Y18LpInUHF}
implies that there is no nonzero \ct{} \hm{}
from $A_1$ to $A_2,$
hence certainly no isomorphism.
\end{proof}

For some values of $p_1$ and~$p_2,$
Theorem~\ref{T_2Y19NFinR} and Theorem~\ref{T_2Y18LpInUHF}
give a lot more.
For example,
if $A_1$ is a spatial $L^{p_1}$~UHF algebra
and $1 \leq p_1 < p_2 \leq 2,$
then there is no nonzero continuous homomorphism from $A_1$
to the bounded operators
on any separable space of the form $L^{p_2} (X, \mu),$
and hence no nonzero continuous homomorphism from $A_1$
to any $L^{p_2}$~UHF algebra, spatial or not.

In the cases in which homomorphisms are not ruled out,
we do not know whether they exist.
We also do not know what happens if the domain $A_1$
is not spatial.

\begin{pbm}\label{Pb_2Y21_NzCtHm}
Do there exist $p_1$ and $p_2$ with $1 \leq p_2 < p_1 \leq 2$
and spatial $L^{p_j}$~UHF algebras $A_j$ for $j = 1, 2$
such that there is a nonzero continuous homomorphism
from $A_1$ to~$A_2$?
What if we drop the requirement that $A_2$ be spatial?
\end{pbm}

\begin{qst}\label{Q_3121_Diffp}
Do there exist distinct $p_1, p_2 \in [1, \I),$
an $L^{p_1}$~UHF algebra~$A,$
a \sfm{} $\XBM,$
and a nonzero \ct{} \hm{} from $A$ to $L (L^{p_2} (X, \mu))$?
\end{qst}

% This might be hard.
A solution might need fairly subtle information about
the Banach space structure of $L^p$~spaces.

\section{Simplicity of $\OP{d}{p}$}\label{Sec:Simplicity}

\indent
The purpose of this section is to prove that
if $d \in \{ 2, 3, 4, \ldots \}$ and $p \in [1, \I),$
then $\OP{d}{p}$ is purely infinite, simple, and amenable.

Recall that
a simple unital \ca{} $A \not\cong \C$ is purely infinite
(that is, every nonzero hereditary subalgebra
contains an infinite projection)
\ifo{} for every nonzero $b \in A$ there are $x, y \in B$
such that $x b y = 1.$
We do not have a theory of hereditary subalgebras of Banach algebras,
so we take the second condition as our definition.

\begin{dfn}\label{D:BanachPI}
We say that a unital Banach algebra $B \not\cong \C$
is {\emph{purely infinite}}
if for every nonzero $b \in B$ there are $x, y \in B$
such that $x b y = 1.$
\end{dfn}

\begin{prp}\label{P:PIIsSimple}
Every purely infinite unital Banach algebra is simple.
\end{prp}

\begin{proof}
The proof is immediate.
\end{proof}

\begin{rmk}\label{R_3218_EqPI}
In Definitions~1.2 of~\cite{AGP},
a ring is defined to be purely infinite
if every nonzero right ideal
contains an infinite idempotent.
(When applied to Banach algebras,
the ideals are not required to be closed.)
By Theorem~1.6 of~\cite{AGP},
and using the fact that $\C$ is the only unital Banach algebra
which is a division ring,
a simple Banach algebra is purely infinite in our sense
\ifo{} it is purely infinite as a unital ring in the sense
of~\cite{AGP}.
\end{rmk}

We recall the following from Section~1 of~\cite{Ph5}
(essentially following Section~1 of~\cite{Cu1}).

\begin{ntn}\label{N:Words}
Let $d \in \{ 2, 3, 4, \ldots \}$
and let $n \in \Nz.$
We define $W_n^d = \{1, 2, \ldots, d \}^n.$
Thus, $W_n^d$ is
the set of all sequences
$\af = \big( \af (1), \af (2), \ldots, \af (n) \big)$
in which $\af (l) \in \{1, 2, \ldots, d \}$
for $l = 1, 2, \ldots, n.$
We set
\[
W_{\I}^d = \coprod_{n = 0}^{\I} W_n^d.
\]
We call the elements of $W_{\I}^d$ {\emph{words}}.
If $\af \in W_{\I}^d,$ the {\emph{length}} of~$\af,$
written $l (\af),$
is the unique number $n \in \Nz$ such that $\af \in W_n^d.$
Note that there is a unique word of length zero,
namely the empty word,
which we write as~$\E.$
For $\af \in W_m^d$ and $\bt \in W_n^d,$
we denote by $\af \bt$ the concatenation,
a word in $W_{m + n}^d.$
\end{ntn}

\begin{ntn}\label{N:WordsInGens}
Let $d \in \{ 2, 3, 4, \ldots \},$
let $n \in \Nz,$ and
let $\af = \big( \af (1), \af (2), \ldots, \af (n) \big) \in W_n^d.$
If $n \geq 1,$ we define $s_{\af}, t_{\af} \in L_d$ by
\[
s_{\af} = s_{\af (1)} s_{\af (2)}
     \cdots s_{\af (n - 1)} s_{\af (n)}
\andeqn
t_{\af} = t_{\af (n)} t_{\af (n - 1)}
     \cdots t_{\af (2)} t_{\af (1)}.
\]
We take $s_{\E} = t_{\E} = 1.$
\end{ntn}

For emphasis:
in the definition of $t_{\af},$
we take the generators $t_{\af (l)}$
{\emph{in reverse order}}.

\begin{lem}[Part of Lemma~1.10 of~\cite{Ph5}]\label{L:PropOfWords}
Let the notation be as in Notation~\ref{N:Words}
and Notation~\ref{N:WordsInGens}.
\begin{enumerate}
\item\label{L:PropOfWords-4}
Let $\af, \bt \in W_{\I}^d.$
Then $s_{\af \bt} = s_{\af} s_{\bt}$
and $t_{\af \bt} = t_{\bt} t_{\af}.$
% \item\label{L:PropOfWords-2}
% Let
% \[
% a_1, a_2, \ldots, a_n
%  \in \{ s_1, s_2, \ldots \} \cup \{ t_1, t_2, \ldots \}.
% \]
% Suppose $a_1 a_2 \cdots a_n \neq 0.$
% Then there exist unique $\af, \bt \in W_{\I}^d$
% such that $a_1 a_2 \cdots a_n = s_{\af} t_{\bt}.$
\item\label{L:PropOfWords-X0}
Let $\af, \bt \in W_{\I}^d$ satisfy $l (\af) = l (\bt).$
Then $t_{\bt} s_{\af} = 1$ if $\af = \bt,$
and $t_{\bt} s_{\af} = 0$ otherwise.
\end{enumerate}
\end{lem}

\begin{lem}[Lemma~1.11 of~\cite{Ph5}]\label{L:mSum}
Let $d \in \{ 2, 3, 4, \ldots \},$
let $L_d$ be as in Definition~\ref{D:Leavitt}, and let $m \in \Nz.$
Then the collection $( s_{\af} t_{\bt} )_{\af, \bt \in W_{m}^d}$
is a system of matrix units for a unital subalgebra of $L_d$
isomorphic to $M_{d^m}.$
That is,
identifying $M_{d^m}$ with the linear maps on a vector space
with basis $W_{m}^d,$
with matrix units $e_{\af, \bt}$ for $\af, \bt \in W_{m}^d,$
there is a unique \hm{}
$\ph_m \colon M_{d^m} \to L_d$
such that $\ph_m (e_{\af, \bt}) = s_{\af} t_{\bt}.$
\end{lem}

\begin{lem}\label{L:LpUHFCore}
Let $d \in \{ 2, 3, 4, \ldots \}$
and let $p \in [1, \I) \setminus \{ 2 \}.$
Let $\OP{d}{p}$ be as in \Def{D:LPCuntzAlg}.
For $m \in \Nz,$
identify $\MP{d^m}{p}$ with $L (l^p (W_{m}^d)).$
Let $(e_{\af, \bt})_{\af, \bt \in W_{m}^d}$ be the standard system
of matrix units in this algebra.
Then there is a unique isometric \hm{}
$\ph_m \colon \MP{d^m}{p} \to \OP{d}{p}$
such that $\ph_m (e_{\af, \bt}) = s_{\af} t_{\bt}.$
Moreover, the $\ph_m$ together define an isometric \hm{}
$\ph$ from a spatial $L^p$~UHF algebra $A$ of type~$d^{\I}$
(see Definitions \ref{D_2Y15_SNat}
and~\ref{D_2Y15_ITP})
to $\OP{d}{p},$
whose range is the closed linear span of
\[
\big\{ s_{\af} t_{\bt} \colon
    {\mbox{$\af, \bt \in W_{\I}^d$ and $l (\af) = l (\bt)$}} \big\}.
\]
\end{lem}

\begin{proof}
% The result is well known for $p = 2.$
% (See 1.5 of~\cite{Cu1}.)
% So assume $p \neq 2.$
%
For $m \in \N,$
let $\ph_m$ be as in Lemma~\ref{L:mSum}.
By definition,
there is a \sfm{} $\XBM$
and a spatial \rpn{} $\rh \colon L_d \to \LLp$
such that we can identify $\OP{d}{p}$ isometrically with
${\overline{\rh (L_d)}}.$
In particular (see Definition 7.4(2) of~\cite{Ph5}),
$\rh (s_j)$ and $\rh (t_j)$ are spatial partial isometries
for $j = 1, 2, \ldots, d.$
Then $\ph_m$ becomes a \rpn{} of $\MP{d^m}{p}$
on $\LLp.$
For $\af, \bt \in W_m^d,$
the standard matrix unit $e_{\af, \bt}$
has image $\ph_m (e_{\af, \bt}) = \rh (s_{\af} t_{\bt}),$
so Lemma 6.17 of~\cite{Ph5}
implies that $\ph_m (e_{\af, \bt})$ is a spatial partial isometry.
This is what it means for $\ph_m$ to be spatial
(Definition~7.1 of~\cite{Ph5}).
It follows from the implication from (1) to~(3)
in Theorem~7.2 of~\cite{Ph5}
that $\ph_m$ is isometric.

Since
\[
\ph_1 (\MP{d}{p}) \subset \ph_2 (\MP{d^2}{p}) \subset \cdots,
\]
it follows that
$A = {\overline{\bigcup_{m = 1}^{\I} \ph_m (\MP{d^m}{p})}}$
is the direct limit of a spatial direct system of type~$d^{\I}.$
So $A$ is a spatial $L^p$~UHF algebra of type~$d^{\I}$
by Theorem~\ref{T_2Y15_SpEx}(\ref{T_2Y15_SpEx_LimIsSp}).
\end{proof}

\begin{prp}\label{P:GaugeAction}
Let $d \in \{ 2, 3, 4, \ldots \}$
and let $p \in [1, \I) \setminus \{ 2 \}.$
Let $\OP{d}{p}$ be as in \Def{D:LPCuntzAlg}.
Then there exists a unique isometric action~$\sm$
of the group $S^1 = \{ \ld \in \C \colon | \ld | = 1 \}$
on $\OP{d}{p}$
such that,
for $\ld \in S^1$ and $j = 1, 2, \ldots, d,$
we have $\sm_{\ld} (s_j) = \ld s_j$
and $\sm_{\ld} (t_j) = \ld^{-1} t_j.$
\end{prp}

\begin{proof}
By Corollary~8.4 of~\cite{Ph5},
there exist a \sfm{} $\XBM$
and a spatial \rpn{} $\rh \colon L_d \to \LLp$
which is free in the sense of Definition 8.1(1) of~\cite{Ph5}.
By Theorem~8.7 of~\cite{Ph5},
we can identify $\OP{d}{p}$ isometrically with
${\overline{\rh (L_d)}}.$
Write $X = \coprod_{k \in \Z} E_k$
as in Definition 8.1(1) of~\cite{Ph5},
with
\[
\rh (s_j) ( L^p (E_k, \mu)) \subset L^p (E_{k + 1}, \mu)
\andeqn
\rh (t_j) ( L^p (E_k, \mu)) \subset L^p (E_{k - 1}, \mu)
\]
for all $k \in \Z$ and for $j = 1, 2, \ldots, d.$

For $\ld \in S^1$ define $g_{\ld} \in L^{\I} (X, \mu)$
by $g_{\ld} (x) = \ld^{k}$ for $k \in \Z$ and $x \in E_k.$
For $f \in L^{\I} (X, \mu),$
let $m (f) \in \LLp$ be the multiplication
operator, defined by
\[
\big( m (f) \xi \big) (x) = f (x) \xi (x)
\]
for
$\xi \in L^p (X, \mu)$
and $x \in X.$
One checks directly that
for $\ld \in S^1$ and for $j = 1, 2, \ldots, d$ we have
\[
m (g_{\ld}) \rh (s_j) m (g_{\ld})^{-1} = \ld \rh (s_j)
\andeqn
m (g_{\ld}) \rh (t_j) m (g_{\ld})^{-1} = \ld^{-1} \rh (t_j).
\]
It follows that there exists a unique automorphism
$\sm_{\ld} \in \Aut \big( \OP{d}{p} \big)$
such that $\sm_{\ld} (s_j) = \ld s_j$
and $\sm_{\ld} (t_j) = \ld^{-1} t_j$ for $j = 1, 2, \ldots, d.$
Since $m (g_{\ld})$ and its inverse are both isometric,
so is $\sm_{\ld}.$
Clearly $\ld \mapsto \sm_{\ld}$ is a group \hm.

It remains only to prove that $\ld \mapsto \sm_{\ld} (a)$
is \ct{} for $a \in \OP{d}{p}.$
This is clearly true for $a \in \rh (L_d),$
and an $\frac{\ep}{3}$~argument extends this result to
all $a \in \OP{d}{p}.$
\end{proof}

\begin{lem}\label{L:FixedPtGauge}
Let $d \in \{ 2, 3, 4, \ldots \}$
and let $p \in [1, \I) \setminus \{ 2 \}.$
Let $\OP{d}{p}$ be as in \Def{D:LPCuntzAlg}.
Let $\sm \colon S^1 \to \Aut \big( \OP{d}{p} \big)$
be the action of Proposition~\ref{P:GaugeAction}.
Identify ${\widehat{S^1}}$ with $\Z$ in the usual way.
For $n \in \Z,$ let $B_n$ be the eigenspace
\[
B_{n} =
 \big\{ b \in \OP{d}{p} \colon
       {\mbox{$\sm_{\ld} (b) = \ld^n b$ for all
                        $\ld \in S^1$}} \big\}
\]
of Proposition~\ref{P:GpEigensp},
and let $P_n$ be as there.
Then:
\begin{enumerate}
\item\label{L:FixedPtGauge-1}
For $\af, \bt \in W_{\I}^d,$
\[
P_n ( s_{\af} t_{\bt} )
  = \begin{cases}
   0               & l (\af) - l (\bt) \neq n
       \\
   s_{\af} t_{\bt} & l (\af) - l (\bt) = n.
\end{cases}
\]
\item\label{L:FixedPtGauge-2}
$B_n$ is the closed linear span of all $s_{\af} t_{\bt}$
with $\af, \bt \in W_{\I}^d$ such that $l (\af) - l (\bt) = n.$
\item\label{L:FixedPtGauge-3}
$B_0$ is the image of the \hm{} $\ph \colon A \to \OP{d}{p}$
of \Lem{L:LpUHFCore}.
\end{enumerate}
\end{lem}

\begin{proof}
Part~(\ref{L:FixedPtGauge-1}) follows from
\begin{equation}\label{Eq:GaugeOnWords}
\sm_{\ld} ( s_{\af} t_{\bt} ) = \ld^{l (\af) - l (\bt)} s_{\af} t_{\bt}.\end{equation}

For~(\ref{L:FixedPtGauge-2}),
let $D_n$ be the closed linear span of all $s_{\af} t_{\bt}$
with $\af, \bt \in W_{\I}^d$ and $l (\af) - l (\bt) = n.$
Part~(\ref{L:FixedPtGauge-1}) implies that if $a \in L_d,$
then $P_n (a) \in D_n.$
By continuity, $P_n (a) \in D_n$ for all $a \in \OP{d}{p}.$
Proposition~\ref{P:GpEigensp}(\ref{P:GpEigensp-3})
now implies that $B_n \subset D_n.$
The reverse inclusion is immediate from~(\ref{Eq:GaugeOnWords}).

Lemma~\ref{L:LpUHFCore} shows that
Part~(\ref{L:FixedPtGauge-3})
is a special case of~(\ref{L:FixedPtGauge-2}).
\end{proof}

We will need the endomorphisms of~$\OP{d}{p}$
in the following proposition.
In the context of Cuntz-Krieger C*-algebras
(where the analogous maps need not be multiplicative),
they are defined before Lemma~2.4 of~\cite{CK}.

\begin{prp}\label{P_2Y21Shift}
Let $d \in \{ 2, 3, 4, \ldots \},$
let $p \in [1, \I) \setminus \{ 2 \},$
and let $r \in \N.$
Define
$\ps_r \colon \OP{d}{p} \to \OP{d}{p}$ by
\[
\ps_r (a) = \sum_{\gm \in W_r^d} s_{\gm} a t_{\gm}
\]
for $a \in \OP{d}{p}.$
Then:
\begin{enumerate}
\item\label{P_2Y21Shift-Endo}
$\ps_r$ is a unital endomorphism of~$\OP{d}{p}.$
\item\label{P_2Y21Shift-Iso}
$\ps_r$ is isometric.
\item\label{P_2Y21Shift-Comm}
For $a \in \OP{d}{p}$ and $\af, \bt \in W_r^d,$
the elements $\ps_r (a)$ and $s_{\af} t_{\bt}$ commute.
\item\label{P_2Y21Shift-Deg0}
With $B_n$ as in Lemma~\ref{L:FixedPtGauge},
we have $\ps_r (B_n) \subset B_n$ for all $n \in \Z.$
\end{enumerate}
\end{prp}

\begin{proof}
In~(\ref{P_2Y21Shift-Endo}),
the only things needing proof are $\ps_r (1) = 1$
and $\ps_r (a b) = \ps_r (a) \ps_r (b)$ for $a, b \in \OP{d}{p}.$
The relation $\ps_r (1) = 1$ follows from Lemma~\ref{L:mSum}.
For the second,
we use Lemma~\ref{L:PropOfWords}(\ref{L:PropOfWords-X0})
at the second step
to get
\[
\ps_r (a) \ps_r (b)
 = \sum_{\gm_1, \gm_2 \in W_r^d}
      s_{\gm_1} a t_{\gm_1} s_{\gm_2} b t_{\gm_2}
 = \sum_{\gm_1 \in W_r^d}  s_{\gm_1} a b t_{\gm_1}
 = \ps_r (a b).
\]

We now prove~(\ref{P_2Y21Shift-Iso}).
For $p = 2,$ the result follows from standard C*-algebra methods.
So assume $p \neq 2.$
Since $\ps_r$ is the $r$-fold composite $\ps_1^r,$
it suffices to consider the case $r = 1.$
We may assume that there is a \sfm{} $\XBM$
such that $\OP{d}{p}$ is the closure
of the range of a spatial \rpn{} of $L_d$ on $L^p (X, \mu).$
Using the conditions in Theorem 7.7(7) and Definition 7.4(1)
of~\cite{Ph5},
we see that $\| s_j \| \leq 1$ and $\| t_j \| \leq 1$
for $j = 1, 2, \ldots, d,$
and that there are
disjoint measurable subsets $X_1, X_2, \ldots, X_d \subset X$
and such that ${\mathrm{ran}} (s_j) \subset L^p (X_j, \mu)$
for $j = 1, 2, \ldots, d.$
Since $\sum_{j = 1}^d s_j t_j = 1,$
we must have $X = \bigcup_{j = 1}^d X_j$
(up to a set of measure zero, which we can ignore)
and ${\mathrm{ran}} (s_j) = L^p (X_j, \mu)$
for $j = 1, 2, \ldots, d.$
Moreover, since the restriction to
$M_d =
\spn \big( \big\{ s_j t_k \colon j, k = 1, 2, \ldots, d \big\} \big)$
(as in the case $m = 1$ of Lemma~\ref{L:mSum})
is a spatial representation of~$M_d,$
it follows from Theorem 7.2(5) of~\cite{Ph5}
that $s_j t_j$ is multiplication
by the characteristic function~$\ch_{X_j}$
for $j = 1, 2, \ldots, d.$

Now let $a \in \OP{d}{p}$ and let $\xi \in L^p (X, \mu).$
Write $\xi = \sum_{j = 1}^d \xi_j$
with $\xi_j \in L^p (X_j, \mu)$ for $j = 1, 2, \ldots, d.$
Then $s_j a t_j \xi = s_j a t_j s_j t_j \xi = s_j a t_j \xi_j$
for $j = 1, 2, \ldots, d.$
Also, whenever $\et_j \in L^p (X_j, \mu)$ for $j = 1, 2, \ldots, d,$
we have
$\big\| \sum_{j = 1}^d \et_j \big\|_p^p
 = \sum_{j = 1}^d \| \et_j \|_p^p.$
So
\begin{align*}
\| \ps_1 (a) \xi \|_p^p
& = \Bigg\| \sum_{j = 1}^d s_j a t_j \xi_j \Bigg\|_p^p
  = \sum_{j = 1}^d \| s_j a t_j \xi_j \|_p^p
  \\
& \leq \sum_{j = 1}^d \| s_j \|^p \cdot \| a \|^p \cdot
        \| t_j \|^p \cdot \| \xi_j \|_p^p
  \leq \| a \|^p \sum_{j = 1}^d \| \xi_j \|_p^p
  = \| a \|^p \cdot \| \xi \|_p^p.
\end{align*}
Thus $\| \ps_1 (a) \| \leq \| a \|.$

We have shown that $\ps_1$ is contractive.
It follows from Proposition~8.9 of~\cite{Ph5} that $\ps_1$ is isometric.

Part~(\ref{P_2Y21Shift-Comm}) is a calculation.
Using Lemma~\ref{L:PropOfWords}(\ref{L:PropOfWords-X0})
at the second and third steps,
for $a \in \OP{d}{p}$ and $\af, \bt \in W_r^d$ we have
\[
s_{\af} t_{\bt} \ps_r (a)
  = \sum_{\gm \in W_r^d} s_{\af} t_{\bt} s_{\gm} a t_{\gm}
  = s_{\af} a t_{\bt}
  = \sum_{\gm \in W_r^d} s_{\gm} a t_{\gm} s_{\af} t_{\bt}
  = \ps_r (a) s_{\af} t_{\bt}.
\]

Part~(\ref{P_2Y21Shift-Deg0}) follows from
the fact that $s_{\gm} \in B_r$ and $t_{\gm} \in B_{- r}$
for $\gm \in W_r^d,$
using several applications
of Proposition~\ref{P:GpEigensp}(\ref{P:GpEigensp-6}).
\end{proof}

\begin{lem}\label{L:PIOnCore}
Let the notation be as in \Lem{L:FixedPtGauge}.
Let $a \in B_0$ be nonzero.
Then there are $n \in \Z,$ $x \in B_{-n},$ and $y \in B_n$
such that $x a y = 1.$
\end{lem}

\begin{proof}
Set $A = B_0,$
and for $m \in \Nz$
let $A_m$ be the linear span of all $s_{\af} t_{\bt}$
with $\af, \bt \in W_{m}^d.$
It follows from \Lem{L:LpUHFCore}
and Lemma~\ref{L:FixedPtGauge}(\ref{L:FixedPtGauge-3})
that $\bigcup_{m = 0}^{\I} A_m$ is dense in~$A,$
and from \Lem{L:LpUHFCore}
and Theorem~\ref{T:LpUHFSimple}
that $A$ is a simple Banach algebra.

By simplicity, there exist $m \in \N$ and
\[
b_1, b_2, \ldots, b_m, c_1, c_2, \ldots, c_m \in A
\]
such that
\[
\sum_{k = 1}^m b_k a c_k = 1.
\]
Set
\[
M = \left( 1 + \sum_{k = 1}^m \| b_k \| \right)
      \left( 1 + \sum_{k = 1}^m \| c_k \| \right).
\]
% Lemma~\ref{L:FixedPtGauge}(\ref{L:FixedPtGauge-2})
% Lemma~\ref{L:FixedPtGauge}(\ref{L:FixedPtGauge-3}),
% and
The density statement above provides $r \in \Nz$ and
\[
a_0, b_1^{(0)}, b_2^{(0)}, \ldots, b_n^{(0)},
   c_1^{(0)}, c_2^{(0)}, \ldots, c_n^{(0)} \in A_r
\]
such that
\[
\| a_0 - a \| < \frac{1}{2 M},
\,\,\,\,\,\,
\left(\sum_{k = 1}^m \big\| b^{(0)}_k \big\| \right)
      \left( \sum_{k = 1}^m \big\| c_k^{(0)} \big\| \right)
  < M,
\]
and
\[
\left\| \sum_{k = 1}^m b_k^{(0)} a_0 c_k^{(0)} - 1 \right\|
   < \frac{1}{2}.
\]
Set
\[
z = \left( \sum_{k = 1}^m b_k^{(0)} a_0 c_k^{(0)} \right)^{-1}
  \in A_r.
\]
Then $\| z \| < 2.$
We have $z b_k^{(0)} \in A_r$ for $k = 1, 2, \ldots, m,$
and
\[
\sum_{k = 1}^m z b_k^{(0)} a_0 c_k^{(0)} = 1.
\]

Choose $n \in \Nz$ such that $d^n \geq m.$
Choose distinct words
\[
\af_1, \af_2, \ldots, \af_m \in W_n^d.
\]
Let $\ps_r \colon \OP{d}{p} \to \OP{d}{p}$
be as in Proposition~\ref{P_2Y21Shift}.
Set $f_{j, k} = \ps_r (s_{\af_j} t_{\bt_k} )$
for $j, k = 1, 2, \ldots, m.$
Also define $s = \ps_r (s_{\af_1})$ and $t = \ps_r (t_{\af_1}).$
Then $s \in B_n$
and $t \in B_{-n}$
by Proposition~\ref{P_2Y21Shift}(\ref{P_2Y21Shift-Deg0}).
Now define
\[
x_0 = t z \sum_{k = 1}^m b_k^{(0)} f_{1, k}
\andeqn
y = \left( \sum_{k = 1}^m f_{k, 1} c_k^{(0)} \right) s.
\]
Using $z, b_k^{(0)}, f_{1, k} \in B_0$
and Proposition~\ref{P:GpEigensp}(\ref{P:GpEigensp-6}),
we get $x_0 \in B_{- n}.$
Similarly $y \in B_n.$

We claim that $x_0 a_0 y = 1.$
To prove this, we compute,
using
Proposition~\ref{P_2Y21Shift}(\ref{P_2Y21Shift-Comm})
and $a_0, b_j^{(0)}, c_k^{(0)} \in A_r$ at the second step
and using Lemma~\ref{L:PropOfWords}(\ref{L:PropOfWords-X0})
and Lemma~\ref{L:mSum} at the third step:
\begin{align*}
x_0 a_0 y
 & = \sum_{j, k = 1}^m
   \ps_r (t_{\af_1}) z b_j^{(0)} \ps_r (s_{\af_1} t_{\bt_j} )
         \ps_r (s_{\af_k} t_{\bt_1} ) c_k^{(0)} \ps_r (s_{\af_1})
     \\
 & = z \sum_{j, k = 1}^m
     b_j^{(0)} a_0 c_k^{(0)}
          \ps_r \big( t_{\af_1} s_{\af_1} t_{\bt_j} s_{\af_k} t_{\bt_1}
                     s_{\af_1} \big)
   = z \sum_{k = 1}^m b_k^{(0)} a_0 c_k^{(0)}
   = 1.
\end{align*}

Applying Proposition~\ref{P_2Y21Shift}(\ref{P_2Y21Shift-Iso}),
we get
\[
\| x_0 \|
   \leq \| z \|
       \sum_{k = 1}^m \big\| b_k^{(0)} \big\|
   \leq 2 \sum_{k = 1}^m \big\| b_k^{(0)} \big\|
\andeqn
\| y \| \leq \sum_{k = 1}^m \big\| c_k^{(0)} \big\|.
\]
So $\| x_0 \| \cdot \| y \| < 2 M.$
Consequently
\[
\| x_0 a y - 1 \|
  = \| x_0 a y - x_0 a_0 y \|
  \leq \| x_0 \| \cdot \| a - a_0 \| \cdot \| y \|
  < 1.
\]
Also $x_0 a y \in B_0$ by
Proposition~\ref{P:GpEigensp}(\ref{P:GpEigensp-6}),
because $x_0 \in B_{- n},$ $a \in B_0,$ and $y \in B_n.$
Therefore $x_0 a y$ has an inverse $w \in B_0.$
Define $x = w x_0,$ which is in $B_{- n}$
by Proposition~\ref{P:GpEigensp}(\ref{P:GpEigensp-6})
and gives $x a y = 1.$
\end{proof}

We need the following result from~\cite{Cu1}.
The important fact about the sequence $\sm$ is that no tail
of it is periodic.

\begin{lem}\label{L:NonperWord}
Let $d \in \{ 2, 3, 4, \ldots, \I \}.$
Let $\sm$ be the sequence
\[
\sm = (1, 2, 1, 1, 2, 2, 1, 1, 1, 2, 2, 2, 1, \ldots),
\]
and for $r \in \N$ let $\sm_r$ be the word in
$W_r^d$ consisting of the first $r$ terms of~$\sm.$
Then for all $m \in \N$ and
\[
\af_1, \af_2, \ldots, \af_m, \bt_1, \bt_2, \ldots, \bt_m \in W_{\I}^d
\]
such that $l (\af_k) \neq l (\bt_k)$ for $k = 1, 2, \ldots, m,$
there exists $r \in \N$ such that
\[
(s_{\sm_r} t_{\sm_r}) s_{\af_k} t_{\bt_k} (s_{\sm_r} t_{\sm_r}) = 0
\]
for $k = 1, 2, \ldots, m.$
\end{lem}

\begin{proof}
Equip $L_d$ with the *-algebra structure of
Lemma 1.6(1) of~\cite{Ph5}.
% \Lem{L:Invol}(\ref{L:Invol-Star}).
Then the statement is a special case of Lemma~1.8 of~\cite{Cu1}.
\end{proof}

\begin{thm}\label{T:AlgIsPI}
Let $d \in \{ 2, 3, 4, \ldots \}$
and let $p \in [1, \I) \setminus \{ 2 \}.$
Then the Banach algebra $\OP{d}{p}$
of \Def{D:LPCuntzAlg} is purely infinite and simple.
\end{thm}

\begin{proof}
By Proposition~\ref{P:PIIsSimple},
it is enough to show that $\OP{d}{p}$ is purely infinite.
Thus, for $a \in \OP{d}{p} \setminus \{ 0 \},$
we must find $x, y \in \OP{d}{p}$ such that $x a y = 1.$
Identify $L_d$ with its image in~$\OP{d}{p}.$

Let the notation be as in \Lem{L:FixedPtGauge}.
We will first find $x$ and $y$ under the assumption that $P_0 (a) = 1.$
Choose $b_0 \in L_d$ such that
$\| a - 1 - b_0 \| < \frac{1}{4}.$
Then it is easy to check
that $P (b_0) \in L_d$ and, since $\| P_0 \| = 1,$
we have $\| P (b_0) \| < \frac{1}{4}.$
Set $b = b_0 - P_0 (b_0).$
Then
\[
b \in L_d,
\,\,\,\,\,\,\
P_0 (b) = 0,
\andeqn
\| a - 1 - b \| < \tfrac{1}{2}.
\]

We claim that that $b$ is a
finite linear combination of elements $s_{\af} t_{\bt}$
for words $\af$ and $\bt$ with $l (\af) \neq l (\bt).$
To see this, use Lemma~1.12 of~\cite{Ph5} to write $b$ as a linear
% To see this, use \Lem{L:SameLength} to write $b$ as a linear
combination
\[
b = \sum_{(\af, \bt) \in F} \ld_{\af, \bt} s_{\af} t_{\bt}
\]
for a finite set $F \subset W_{\I}^d \times W_{\I}^d$
and with $\ld_{\af, \bt} \in \C$ for $(\af, \bt) \in F.$
Since $P_0 (b) = 0,$
\Lem{L:FixedPtGauge}(\ref{L:FixedPtGauge-1})
allows us to simply omit all the terms corresponding to
pairs $(\af, \bt)$ with $l (\af) = l (\bt).$
The claim follows.

\Lem{L:NonperWord} therefore provides a word $\gm \in W_{\I}^d$
(namely $\sm_r$ for some~$r \in \N$)
such that $(s_{\gm} t_{\gm}) b (s_{\gm} t_{\gm}) = 0.$
Since $t_{\gm} s_{\gm} = 1$
(by \Lem{L:PropOfWords}(\ref{L:PropOfWords-X0})),
we get
\[
t_{\gm} b s_{\gm}
 = t_{\gm} s_{\gm} t_{\gm} b s_{\gm} t_{\gm} s_{\gm}
 = 0.
\]
Therefore
\[
\| t_{\gm} a s_{\gm} - 1 \|
  = \| t_{\gm} (a - 1) s_{\gm} \|
  = \| t_{\gm} (a - 1 - b) s_{\gm} \|
  \leq \| a - 1 - b \|
  < \tfrac{1}{2}.
\]
It follows that $t_{\gm} a s_{\gm}$ is invertible.
Take $x = t_{\gm}$ and $y = s_{\gm} (t_{\gm} a s_{\gm})^{-1}.$
This proves the case $P_0 (a) = 1.$

Now instead of assuming $P_0 (a) = 1,$
we merely assume $P_0 (a) \neq 0.$
\Lem{L:PIOnCore} provides $n \in \Nz,$
$x_0 \in B_{- n},$ and $y_0 \in B_n$ such that $x_0 P_0 (a) y_0 = 1.$
Two applications of Proposition~\ref{P:GpEigensp}(\ref{P:GpEigensp-7})
show that $P_0 (x_0 a y_0) = x_0 P_n (a y_0) = x_0 P_0 (a) y_0.$
By the case already done,
there exist $x_1, y_1 \in \OP{d}{p}$
such that $x_1 x_0 a y_0 y_1 = 1.$
So take $x = x_1 x_0$ and $y = y_0 y_1.$

Finally, we treat the general case.
Proposition~\ref{P:GpEigensp}(\ref{P:GpEigensp-9})
implies that there is $n \in \N$ such that $P_n (a) \neq 0.$
If $n = 0,$
the previous case applies.
If $n > 0,$
then Proposition~\ref{P:GpEigensp}(\ref{P:GpEigensp-7})
implies that $P_0 (a t_1^n) = P_n (a) t_1^n.$
Also,
\[
P_0 (a t_1^n) s_1^n = P_n (a) t_1^n s_1^n = P_n (a) \neq 0,
\]
so the previous case provides $x_0, y_0 \in \OP{d}{p}$
such that $x_0 a t_1^n y_0 = 1.$
Take $x = x_0$ and $y = t_1^n y_0.$
If $n < 0,$ a similar argument
provides $x_0, y_0 \in \OP{d}{p}$
such that $x_0 s_1^{-n} a y_0 = 1,$
and we can take $x = x_0 s_1^{-n}$ and $y = y_0.$
\end{proof}

We state a consequence related to K-theory.
It will become more relevant once we have computed
the K-theory of~$\OP{d}{p},$
which will be done in~\cite{Ph7}.
Recall (see the beginning of Section~1 of~\cite{AGP})
that idempotents $e$ and $f$ in a ring~$A$
are (algebraically) Murray-von Neumann equivalent
if there exist $x \in e A f$ and $y \in f A e$
such that $x y = e$ and $y x = f.$
(It is equivalent to merely require
that there be $v, w \in A$ such that $v w = e$ and $w v = f.$
Indeed, given $v$ and~$w,$
set $x = e v f$ and $y = f w e.$
Then
\[
x y = (e v f) (f w e) = (v w v w v) (w v w v w) = e^5 = e,
\]
and similarly $y x = f.$)

\begin{cor}\label{C_3210_PIKthy}
Let $d \in \{ 2, 3, 4, \ldots \}$
and let $p \in [1, \I) \setminus \{ 2 \}.$
Then $K_0 \left( \OP{d}{p} \right)$
is the set of algebraic Murray-von Neumann equivalence classes
of nonzero idempotents in $\OP{d}{p}$
with the group operation given by $[e] + [f] = [s_1 e s_1 + s_2 f s_2].$
\end{cor}

\begin{proof}
In view of Remark~\ref{R_3218_EqPI} and Theorem~\ref{T:AlgIsPI},
the result
can be deduced from
Lemma~1.4, Proposition~2.1, and Corollary~2.2 of~\cite{AGP},
together with the identification of isomorphism classes
of projective modules with (algebraic Murray-von Neumann)
equivalence classes of idempotents,
as described at the beginning of Section~2 of~\cite{AGP}.
\end{proof}

The following corollary is motivated by the fact
(Theorem~7.2 of~\cite{AGOP}) that a simple unital \ca{}
has real rank zero \ifo{} it is an exchange ring,
in the sense described at the beginning of Section~1 of~\cite{AGOP}.

\begin{cor}\label{C_3218_Exch}
Let $d \in \{ 2, 3, 4, \ldots \}$
and let $p \in [1, \I) \setminus \{ 2 \}.$
Then $\OP{d}{p}$ is an exchange ring.
\end{cor}

\begin{proof}
Apply Theorem 1.1 of~\cite{Ar},
Theorem~\ref{T:AlgIsPI},
and Remark~\ref{R_3218_EqPI}.
\end{proof}

The only thing left to do is to prove that $\OP{d}{p}$ is amenable.
The following theorem is essentially due to Rosenberg~\cite{Rs}.
(We are grateful to Narutaka Ozawa for calling our attention
to this reference.)

\begin{thm}\label{T_2Y22_RosAmen}
Let $B$ be a unital Banach algebra,
let $A \subset B$ be a closed subalgebra of~$B,$
and let $s, t \in B.$
Suppose:
\begin{enumerate}
\item\label{T_2Y22_RosAmen_Iso}
$t s = 1.$
\item\label{T_2Y22_RosAmen_Inv}
$s A t \subset A.$
\item\label{T_2Y22_RosAmen_Norm}
$\| s \| \leq 1$ and $\| t \| \leq 1.$
\item\label{T_2Y22_RosAmen_Gen}
$A,$ $s,$ and $t$ generate $B$ as a Banach algebra.
\item\label{T_2Y22_RosAmen_AA}
$A$ is amenable as a Banach algebra.
\end{enumerate}
Then $B$ is amenable as a Banach algebra.
\end{thm}

\begin{proof}
When $B$ and $A$ are \ca{s} and $t = s^*,$
this is Theorem~3 of~\cite{Rs}.
The proof given in~\cite{Rs} also works under the hypotheses
of our theorem.
\end{proof}

\begin{cor}\label{C_2Y22_OdpAmen}
Let $d \in \{ 2, 3, 4, \ldots \}$
and let $p \in [1, \I) \setminus \{ 2 \}.$
Then the Banach algebra $\OP{d}{p}$
of \Def{D:LPCuntzAlg} is amenable as a Banach algebra.
\end{cor}

\begin{proof}
We verify the hypotheses of Theorem~\ref{T_2Y22_RosAmen}.
Let $A \subset \OP{d}{p}$
be the image of the \hm~$\ph$
of Lemma~\ref{L:LpUHFCore}.
Take $s = s_1$ and $t = t_1.$

Condition~(\ref{T_2Y22_RosAmen_Iso}) of Theorem~\ref{T_2Y22_RosAmen}
is contained in Definition~\ref{D:Leavitt}(\ref{D:Leavitt_1}).
Condition~(\ref{T_2Y22_RosAmen_Inv})
% of Theorem~\ref{T_2Y22_RosAmen}
follows from the definition of~$A$
in Lemma~\ref{L:LpUHFCore}.
Condition~(\ref{T_2Y22_RosAmen_Norm})
% of Theorem~\ref{T_2Y22_RosAmen}
is part of Definition~\ref{D_2Y14_SpRep}.
Condition~(\ref{T_2Y22_RosAmen_AA})
% of Theorem~\ref{T_2Y22_RosAmen}
is Theorem~\ref{T_3714Amen}.

For
Condition~(\ref{T_2Y22_RosAmen_Gen}),
% of Theorem~\ref{T_2Y22_RosAmen}
we need only show that the the subalgebra $D$ of $\OP{d}{p}$
generated by $A,$ $s,$ and $t$
contains $s_1, s_2, \ldots, s_d, t_1, t_2, \ldots, t_d.$
For $j, k = 1, 2, \ldots, d,$
we have $s_j t_k \in A.$
For $j = 1, 2, \ldots, d$ we therefore get
\[
s_j = s_j t_1 s_1 = (s_j t_1) s \in D
\andeqn
t_j = t_1 s_1 t_j = t (s_1 t_j) \in D.
\]
This completes the proof.
\end{proof}

\end{document}